%% file: text10.tex
\documentclass[12pt,reqno]{amsart}
\usepackage[headings]{fullpage}
\usepackage{amssymb,amsmath,amscd,bbm,tikz-cd}
\usepackage[all,cmtip]{xy}
\usepackage{url}
\usepackage[bookmarks=true,%
    colorlinks=true,% 
    linkcolor=blue,%\part{title}
    citecolor=blue,%
    filecolor=blue,%
    menucolor=blue,%
    urlcolor=blue,%
    breaklinks=true]{hyperref}
\usepackage{slashed}
\usepackage{listings}
\usepackage{verbatim}
\usepackage{mathtools}
\usepackage[normalem]{ulem}   % \sout{mytext} crosses out my text
\usepackage{graphicx}
\usepackage{texdraw}
\usepackage{caption}
\graphicspath{{figures/}}

\newtheorem{theorem}{Theorem}[section]
\theoremstyle{definition}

\newtheorem{lemma}[theorem]{Lemma}
\newtheorem{definition}[theorem]{Definition}
\newtheorem{remark}[theorem]{Remark}

\newtheorem{example}[theorem]{Example}

\def\red#1{\textcolor[rgb]{1.00,0.00,0.00}{#1}}
\def\blue#1{\textcolor[rgb]{0.00,0.00,1.00}{#1}}

\def\BZ{\mathbbm Z}

\def\BR{\mathbbm R}
\def\BC{\mathbbm C}

\def\calO{\mathcal O}

\def\calT{\mathcal T}

\def\li{\mathrm{Li}_2}

\def\be{\begin{equation}}
\def\ee{\end{equation}}

\renewcommand{\Ddot}{\mathring{D}}
\newcommand{\curly}{\mathrel{\leadsto}}

\begin{document}
	
\title[1-loop equals torsion for two-bridge knots]{1-loop equals torsion for
  two-bridge knots}
\author{Stavros Garoufalidis}
\address{% Max Planck Institute for Mathematics \\
% Vivatsgasse 7, 53111 Bonn, GERMANY \newline
International Center for Mathematics, Department of Mathematics \\
Southern University of Science and Technology \\
Shenzhen, China
%and Max Planck Institute for Mathematics \\
%Bonn, Germany 
\newline
{\tt \url{http://people.mpim-bonn.mpg.de/stavros}}}
\email{stavros@mpim-bonn.mpg.de}

\author{Seokbeom Yoon}
\address{Department of Mathematics \\
Chonnam National University \\
Gwangju, South Korea \newline
{\tt \url{http://sites.google.com/view/seokbeom}}}
\email{sbyoon15@gmail.com}


\date{13 November, 2024}%\today}

\begin{abstract}
  Motivated by the conjectured asymptotics of the Kashaev invariant, Dimofte and
  the first author introduced a power series associated to a suitable ideal
  triangulation of a cusped hyperbolic 3-manifold, proved that its constant
  (1-loop) term is a topological invariant and conjectured that it equals to the
  adjoint Reidemeister torsion. We prove this conjecture for hyperbolic 
  2-bridge knots by combining the work of Ohtsuki--Takata with an explicit computation.
\end{abstract}

\maketitle
{
\footnotesize
\tableofcontents
}

\section{Introduction}
\label{sec.intro}

The celebrated volume conjecture of Kashaev predicts that the growth rate of the
synonymous invariant of a hyperbolic knot detects the the volume of the knot
complement \cite{Kas95}. This, combined with the result of Murakami--Murakami
that the Kashaev invariant is given by an evaluation of the colored Jones polynomial
~\cite{MM01} gives a deep connection between the Jones polynomial of a knot in 3-space
and hyperbolic geometry. The volume conjecture can be extended to a stronger
statement concerning the asymptotics of the Kashaev invariant to all orders in
perturbation theory~\cite{GL:asy,DGLZ}, and a natural question was to give a direct
definition of the corresponding power series. This was the motivation of~\cite{DG1},
where Dimofte and the first author introduced a power series associated to a
suitable ideal triangulation of a cusped hyperbolic 3-manifold, proved that its
constant (1-loop) term is a topological invariant and conjectured that it equals
to the adjoint Reidemeister torsion. 

More precisely, the 1-loop invariant depends on an ideal triangulation that detects
the geometric representation of a cusped hyperbolic 3-manifold $M$ (call such
triangulations essential), and it is an element of the invariant trace field of
$M$, well-defined up to a sign.
Moreover, it is unchanged under 2--3 and 0--2 Pachner moves of essential ideal
triangulations (see e.g.,~\cite[Sec.3]{DG1} and~\cite[Prop.5.1]{PW2023}).
Since every cusped hyperbolic 3-manifold has an essential ideal triangulation
obtained by subdividing the Epstein--Penner ideal cell decomposition, and the set
of essential ideal triangulations is connected under 2--3 and 0--2 Pachner moves,
as shown by Kalelkar--Schleimer--Segerman~\cite{KSS24}, it follows that  
the 1-loop is an invariant of a cusped hyperbolic 3-manifold. The conjecture
of~\cite{DG1} is that it equals (up to multiplication by a sign) to the adjoint
Reidemeister torsion. This conjecture is known for fibered cusped hyperbolic
3-manifolds \cite{DGY2023} and for fundamental shadow links~\cite{PW2023}
(see also \cite{AW2024}). Our goal is to prove that it also holds for hyperbolic
2-bridge knots.

\begin{theorem}
\label{thm.main}
The 1-loop equals torsion conjecture holds for hyperbolic 2-bridge knots.
\end{theorem}

Let us comment on this conjecture and our proof for hyperbolic 2-bridge knots.
Both the 1-loop invariant and the adjoint Reidemeister torsion are given (up to
normalization factors) by the determinant of a matrix with coefficients in the
trace field. For the 1-loop, the size of the matrix is the number of tetrahedra
of an essential ideal triangulation, and the matrix is obtained by the Neumann--Zagier
matrices of the ideal triangulation and their shapes. Alternatively, as was shown by
the second author, the 1-loop is essentially the Jacobian of Ptolemy
equations of an essential ideal triangulation~\cite{Yoon2024}. On the other hand, 
a matrix for the adjoint Reidemeister torsion can be obtained from a presentation of
the fundamental group of the 3-manifold $M$ by applying Fox calculus first on
a relator set of $\pi_1(M)$, and then replacing $\pi_1(M)$ by the adjoint
representation of $\mathfrak{sl}_2(\BC)$ using
the geometric representation. An ideal triangulation does give a presentation
of the fundamental groupoid of $M$, and after further choices, of $\pi_1(M)$.
Thus, both the 1-loop and the adjoint Reidemeister torsion can be defined from
an essential ideal triangulation, but their definitions have different origins and
are not exactly compatible. 

To prove our Theorem~\ref{thm.main}, we use an alternative approach
from~\cite{DGY2023}. Starting from a planar projection of a 2-bridge knot,
we consider the corresponding octahedral decomposition of its complement minus two
points, then collapse some tetrahedra to obtain an ideal triangulation of its
complement (see Section~\ref{sec.octa} for details).
We then eliminate a certain number of variables and prove by an explicit
calculation that the corresponding Jacobian is given by the Ohtsuki--Takata invariant
of the initial planar projection (see Section~\ref{sec.pf} for details).
Using Ohtsuki--Takata's theorem, \cite[Thm.1.1]{OT2015} for hyperbolic 2-bridge knots,
we conclude the proof of Theorem~\ref{thm.main}.

In theory, the proof of Theorem~\ref{thm.main} should apply to essential planar
projections of hyperbolic knots, but the intermediate calculations are not clear
to us, and more fundamentally, 

\begin{itemize}
\item[(a)]
  It is not known that every hyperbolic knot has an essential open planar diagram.
\item[(b)]
  Although every two planar diagrams of a knot are connected by Reidemeister moves,
  it is not known if this holds for the set of essential open planar diagrams, nor
  that there is a canonical connected component of that set. 
\end{itemize}
In contrast, if we replace essential planar diagrams with essential ideal
triangulations, both problems are solved. For (a), one can use a subdivision of
the canonical ideal cell decomposition of a cusped hyperbolic 3-manifold. For (b)
one can use a canonical connected component, namely the set of subdivisions of the
canonical ideal cell decomposition~\cite{DG1}, and even better, it is now known
that the set of essential ideal triangulations is connected~\cite{KSS24}.

On the positive side, the standard diagrams of hyperbolic 2-bridge knots are
essential~\cite{OT2015}, and more generally, the alternating reduced
diagrams of alternating hyperbolic knots are essential~\cite{GMT,SY}. Note that
2-bridge knots are alternating, and all of them, with the exception of $(2,b)$-torus
knots, are hyperbolic.

%%%%%%%%%%%%%%%%%%%%%%%%%%%%%%%%%%%%%%%%%%%%%%%%%%%%%%%%%%%%%%%%%%%%%%%%%%%%
%%%%%%%%%%%%%%%%%%%%%%%%%%%%%%%%%%%%%%%%%%%%%%%%%%%%%%%%%%%%%%%%%%%%%%%%%%%%

\section{Two invariants from the asymptotics of the Kashaev invariant}
\label{sec.prelim}

In this section we recall the 1-loop and Ohtsuki--Takata invariants, 
following \cite{DG1} and \cite{OT2015}. They are both expected to equal to each other
and to be the constant term of the asymptotic expansion of the Kashaev invariant, but
their definitions are a bit different. The 1-loop invariant is modeled on Chern--Simons
perturbation theory, and is defined using an essential ideal triangulation of a
knot complement, whereas the Ohtsuki--Takata invariant depends on a planar projection
of a knot and the invariant is obtained by applying stationary phase to a state-sum
formula the Kashaev invariant.

\subsection{The 1-loop invariant}
\label{sec.OL}

The 1-loop invariant is a function
\be
\tau : \{ \text{essential triangulations} \} \to \BC^\times
\ee
that we now recall following~\cite{DG1}. We fix an essential ideal triangulation
$\calT$ of an oriented, cusped hyperbolic 3-manifold $M$ (such as the complement
of a hyperbolic knot in $S^3$) and denote the edges and tetrahedra of $\calT$ by
$e_i$ and by $\Delta_j$, respectively, for $1 \leq i,j \leq N$. Note that the number of
edges is equal to that of tetrahedra, as $M$ has the Euler
characteristic 0. We fix a quad type of each tetrahedron $\Delta_j$. This 
means that each edge of $\Delta_j$ is assigned to a shape parameter among 
\begin{equation*}
z_j, \quad z'_j:=\frac{1}{1-z_j}, \quad \textrm{or} \quad z''_j:=1- \frac{1}{z_j}
\end{equation*}
with opposite edges having same parameters as in Figure~\ref{fig.tetrahedron}. 

\begin{figure}[htpb!]
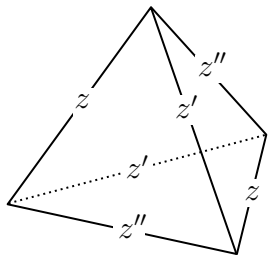
\caption{An ideal tetrahedron with shape parameters.}
\label{fig.tetrahedron}
\end{figure}

The \emph{gluing equation matrices} $G, G'$ and $G''$  are $N \times N$ integer
matrices whose rows and columns are indexed by the edges and by the tetrahedra
of $\calT$,  respectively. The $(i,j)$-entry of $G$ (resp., $G'$ and $G''$) is 
the number of edges of $\Delta_j$ with the shape parameter $z_j$ (resp., $z'_j$
and $z''_j$) that are identified with $e_i$. As the name suggests, they 
determine the gluing equations of $\calT$: the gluing equation of an edge $e_i$
is given as 
\be
\label{eqn.e}
e_i :  \quad \sum_{j=1}^N  G_{ij} \log z_j + G'_{ij} \log z'_j
+ G''_{ij} \log z''_j = 2 \pi i \, .
\ee 
It is known that the gluing equation matrices have redunancy, and we replace
one of their rows by using a meridian $\mu$ of $K$. Precisely, let $(C,C',C'')$ 
be a triple of row vectors in $\BZ^N$ that describe the completeness equation
of $\mu$ as
\be
\label{eqn.mu}
\mu : \quad \sum_{j=1}^N C_{j}\log z_j +C'_{j} \log z'_j
+ C''_{j} \log z''_j  = 0 \, .
\ee
We replace one row of $G$, $G'$ and $G''$ with $C$, $C'$ and $C''$, respectively,
and denote by $G_\mu$, $G'_\mu$ and $G''_\mu$ the resulting matrices.

One auxiliary ingredient for defining the 1-loop invariant is a 
\emph{flattening}. It is a triple of column vectors $f, f', f'' \in \BZ^N$
satisfying
\begin{align*}
f+ f' +  f'' & =(1,\ldots,1)^t \,, \\
G f + G' f' +G'' f'' & =(2,\ldots,2)^t \, .
\end{align*}
Note that a flattening in \cite{DG1} requires one additional condition,
but it turned out to be dispensable \cite{Yoon2024}.

We now come to the assumption that $M$ is hyperbolic and $\calT$ is essential. 
This allows us to find a 
geometric solution $z^\circ=(z^\circ_1,\ldots, z^\circ_N )$
of $\calT$, whose holonomy representation is geometric. Here a solution
means  a tuple of complex numbers, other than 0 and 1, satisfying 
the gluing equation~\eqref{eqn.e} for all edges and the completeness
equation~\eqref{eqn.mu}. 

\begin{definition}[\cite{DG1}]
The \emph{1-loop invariant} of an essential triangulation $\calT$ is defined as
\begin{equation*}
\tau(\calT)  :=  \pm
\frac{ \det \left( G_\mu\, \mathrm{diag}(\zeta)	+G_\mu' \, \mathrm{diag}(\zeta')
+ G_\mu'' \, \mathrm{diag}(\zeta'') \right)}{2 \!\!
\displaystyle\prod_{1 \leq j \leq N} \! \!
\zeta_j^{f_j} {\zeta_j'}^{f_j'} {\zeta_j''}^{f_j''}} 
\end{equation*}
where the right-hand side is evaluated at the geometric solution $z^\circ$. Here
\begin{equation*}
\zeta_j := \frac{d \log z_j}{dz_j}=\frac{1}{z_j}, \quad
\zeta_j' := \frac{d \log z_j'}{dz_j}=\frac{1}{1-z_j},\quad
\zeta_j'' := \frac{d \log z_j''}{dz_j}=\frac{1}{z_j(z_j-1)}   \, ,
\end{equation*}
and
$\mathrm{diag}(\zeta^\square)$ is the diagonal matrix with diagonal entries
$\zeta^\square_1,\ldots,\zeta^\square_N$ for  $\square \in \{ \ , ' ,''\}$.
\end{definition}

\begin{remark}
\label{rmk.der}
The above definition is a symmetric version of the original one given in \cite{DG1}
and was introduced by Siejakowski \cite{Sie21}. It is worth noting that  
\begin{equation*}
\det \left( G_\mu\, \mathrm{diag}(\zeta)+G_\mu' \, \mathrm{diag}(\zeta')+ G_\mu'' \,
\mathrm{diag}(\zeta'') \right) = \det
\left( \dfrac{\partial (g_1,\ldots, g_{N-1}, g_\mu)} {\partial (z_1,\ldots,z_N)}\right)
\end{equation*}
where $g_i$ (resp., $g_\mu$) refers to the left-hand side of
\eqref{eqn.e} (resp., \eqref{eqn.mu}).
\end{remark} 

It was proved in~\cite{DG1} that the 1-loop invariant $\tau(\calT)$ does not depend
on the choice of a flattening (hence it is omitted in the notation) and that it is
unchanged under Pachner 2--3 moves and 0--2 on essential triangulations;
see~\cite[Sec.3]{DG1} for the 2--3 move and~\cite[Prop.5.1]{PW2023} for the
historically ommitted 0--2 move. Since every cusped hyperbolic
manifold has an essential triangulation, (obtained for instance by subdividing the
Epstein--Penner ideal cell decomposition as was used in~\cite{DG1}), and the set
of essential triangulations is connected under Pachner 2--3 and 0--2 moves, it
follows that it is a topological invariant. By its very definition, the 1-loop
invariant is an element of the invariant trace field of the cusped hyperbolic
3-manifold,well-defined up to multiplication by a sign.

\begin{example}
\texttt{SnapPy}'s default triangulation $\calT$ of the knot $6_1$ consists of
4 tetrahedra with gluing equation matrices
\begin{equation*}
G_\mu = 
\begin{pmatrix}
1 & 1 & 0 & 0  \\
0 & 0 & 0 & 1  \\
0 & 1 & 1 & 0  \\
-1 & 0 & 0 & 0 
\end{pmatrix}, \
G'_\mu =
\begin{pmatrix}
0 & 2 & 0 & 1  \\
1 & 0 & 1 & 0  \\
1 & 0 & 0 & 0  \\
1 & 0 & 0 & 0   
\end{pmatrix}, \
G''_\mu = 
\begin{pmatrix}
1 & 0 & 1 & 1  \\
1 & 2 & 1 & 0  \\
0 & 0 & 0 & 1  \\
0 & 1 & 0 & 0  
\end{pmatrix}
\, .
\end{equation*}
The 1-loop invariant, evaluated at the geometric solution
\begin{align*}
z^\circ \approx \ & (0.89152 -1.55249i ,\, 0.043315 - 0.64120i, 
\ -1.50411-1.22685i, \, 0.17385 -1.06907i)
\end{align*}
with a flattening 
\begin{equation*}
(f,f',f'')= ( (0,1,0,0)^t, \ (1,0,1,1)^t, \ (0,0,0,0)^t) \,,
\end{equation*}
is given by $\tau(\calT) \approx  0.487465 + 1.738045 i$. This agrees with
the adjoint Reidemeister torsion of the knot $6_1$.
\end{example}

\subsection{Essential open diagrams and their potential}
\label{sub.PD}

The domain of the Ohtsuki--Takata invariant discussed below is the set of
essential open diagrams of knots. An open diagram is a planar diagram 
of a $(1,1)$-tangle and is said to be essential if the corresponding 
collapsed triangulation is essential. We postpone a detailed description of
the collapsed ideal triangulation to Section~\ref{sec.octa}. Roughly, this is
a story of how to pass from a planar diagram of a knot to an ideal triangulation of
its complement and use it to determine the complete hyperbolic structure of a
hyperbolic knot. Like so many things in hyperbolic geometry, this method goes
back to Thurston.

We here recall a potential function associated to an open diagram, introduced by
Yokota \cite{Yokota02} and motivated by the $R$-matrix state-sum formulas for
the Kashaev invariant of a knot.

In what follows, we fix an open diagram $\Ddot$ of a hyperbolic knot $K$ in $S^3$,
that is, a planar projection of a $(1,1)$-tangle whose closure is $K$. For technical
reasons, we will assume that 
\begin{center}
$(\dagger)$
starting from one endpoint of $\Ddot$, we overpass the first crossing, and
from the other endpoint, we underpass the first crossing.
\end{center}

Viewing $\Ddot$ as an embedded 4-valent graph with two univalent vertices
in $\BR^2$, 
\begin{itemize}
\item a \emph{segment} refers to an edge of the graph that is not adjacent
to a univalent vertex,
\item a \emph{region} refers to a connected component of the complement of
the graph, and 
\item a \emph{corner} refers to a fan-shaped area around a crossing separated
by segments. 
\end{itemize}

We assign the constant $1$ to every segment of $\Ddot$ adjacent to the unbounded
region. Also, starting from an endpoint of $\Ddot$, if we underpass (resp., overpass) 
the first crossing, we assign the constant $\infty$ (resp., $0$) to each segment
until we encounter an overpass (resp., underpass). To each segment not
assigned a constant, we assign a variable so that every segment receives either
a constant or a variable. For a non-alternating diagram, consecutive segments
could be assigned with $\infty$. If it is the case, for each crossing lying
between such segments, we set two adjacent variables, other than $\infty$,
to be equal. We do the same for consecutive segments assigned with $0$. 

\begin{example}
\label{ex.61}  
An essential open diagram for the $6_1$ knot, along with an assignment
of variables to its segments is shown in Figure~\ref{fig.example}. 
\end{example}

\begin{figure}[htpb!]
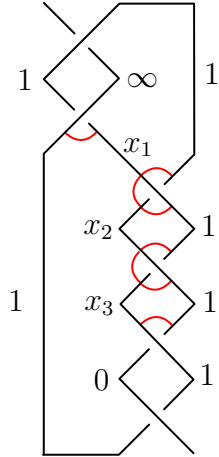
\caption{An open diagram of the knot $6_1$ with 3 variables $x_1$, $x_2$,
$x_3$ and 8 essential corners marked in red.}
\label{fig.example}
\end{figure}

We now define the potential function $V$ of $\Ddot$, following Yokota.
We associate a dilogarithm function with each corner as in
Figure~\ref{fig.crossing}.  Such dilogarithms could be ill-defined or constant.
This happens precisely when either $x$ or $y$ in Figure~\ref{fig.crossing} is
$0$ or $\infty$, or when both $x$ and $y$ are $1$. We say that a corner is
\emph{essential} if it is not the case, i.e. if the associated dilogarithm is
well-defined and non-constant. The \emph{Yokota potential function} $V$ is 
defined as the sum of those dilogarithms over all essential corners. 
\begin{figure}[htpb!]
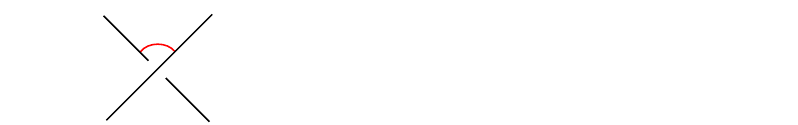
\caption{Dilogarithm functions associated with corners.}
\label{fig.crossing}
\end{figure}

We now discuss a key point, namely an isomorphism of the critical points of the
potential function $V$ with the solutions to the gluing and completeness
equations of the collapsed triangulation $\calT_{\Ddot}$; see~\cite[Thm.4.1]{MY2018}.
Here a critical point of $V$ means a point $x=(x_1,\ldots,x_n) \in \BC^n$
satisfying
\begin{equation}
\label{eqn.crit}
\exp  \left ( x_i \, \dfrac{\partial V}{\partial x_i} \right) = 1,
\qquad i=1,\dots, n\,.
\end{equation}
Under this isomorphism,
the essential corners of $\Ddot$ are in bijection with the tetrahedra of
$\calT_{\Ddot}$, and the the variables of the Yokota potential function are related
to the shapes of the tetrahedra of $\calT_{\Ddot}$ according to
Figure~\ref{fig.crossing2}. In particular, if $\Ddot$ is essential, there is a
geometric solution $x^\circ=(x^\circ_1,\ldots,x^\circ_n)$ to Equation~\eqref{eqn.crit}
which gives rise to the complete hyperbolic structure of $\calT_{\Ddot}$.

\begin{figure}[htpb!]
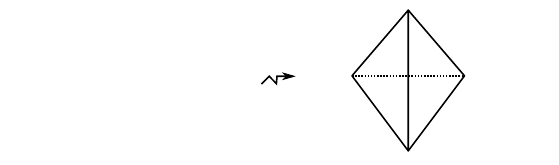
\caption{An ideal tetrahedron at an essential corner.}
\label{fig.crossing2}
\end{figure}

\noindent
{\bf Example~\ref{ex.61} continued.} 
The diagram in Figure~\ref{fig.example} has 8 essential corners marked by red
and the Yokota potential function is given by
\begin{align*}
V & =  \   \li \left( x_1 \right) -\li \left(\dfrac{1}{x_1} \right) + 
\li \left(\dfrac{x_2}{x_1} \right) - \li \left( x_2 \right) \\
& \qquad -  \li \left(\dfrac{1}{x_2} \right) +  \li \left(\dfrac{x_3}{x_2} \right) - 
\li \left(x_3 \right) - \li \left(\dfrac{1}{x_3} \right) \, .
\end{align*}
The critical point equations~\eqref{eqn.crit} are given by
\begin{equation*}
\frac{-x_1+x_2}{(-1+x_1)^2} =1, \, \frac{x_1(x_2-x_3)}{-x_1 + x_2} =1, \,
\frac{x_2 x_3}{-x_2 + x_3} =1 \,,
\end{equation*}
%% see Mathematica file: yoon/CriticalPoints.Potential.61.nb
and equivalently, in the form 
\be
2 - 5 x_1 + 6 x_1^2 - 3 x_1^3 + x_1^4 =0, \qquad x_2 = 1 - x_1 + x_1^2, \qquad
x_3 = \frac{1}{2}(1 + 2 x_1 - x_1^2 + x_1^3)
\ee
and the geometric solution is approximately 
\be
\label{61xapp}
x^\circ \approx \,\, (0.895123 -1.552491 i , \, -1.504108 - 1.226851 i, \, 
-0.677958 -0.157779i) \, .
\ee

\subsection{The Ohtsuki--Takata invariant}
\label{sub.OT}

To define the Ohtsuki--Takata invariant, we consider two more functions
associated with $\Ddot$. To define the first one, 
for each crossing, cup and cap of $\Ddot$, we choose two corners
around it and associate a rational function with each of them, shown as in
Figure~\ref{fig.omega1}. These  functions are well-defined and non-zero
for essential corners. The \emph{first normalizing function} $\Omega_1$ is
defined as the product of these rational functions over all essential corners
among chosen ones. 

\begin{figure}[htpb!]
	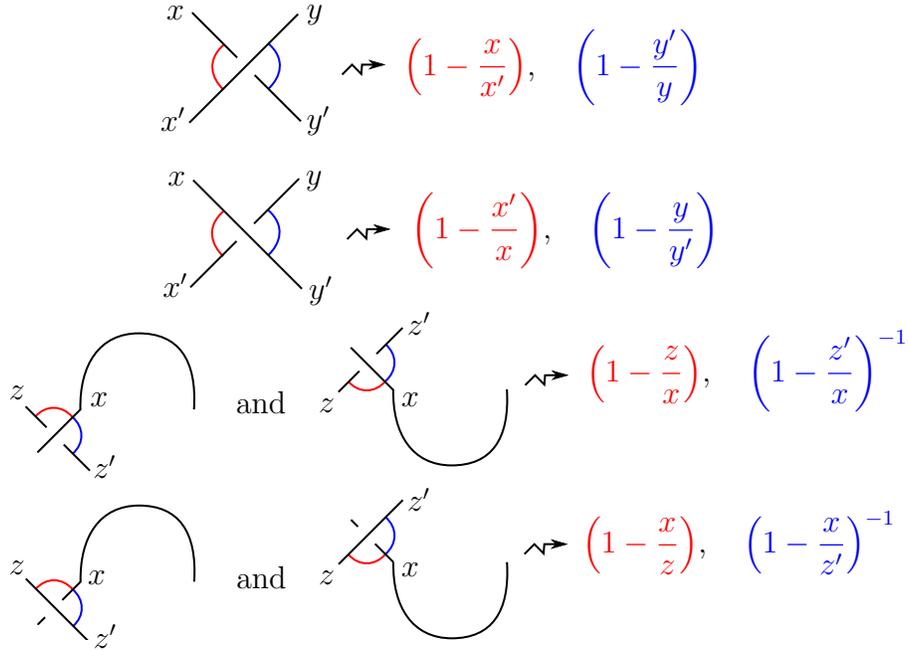
	\caption{Rational functions associated with crossings, cups and caps.}
	\label{fig.omega1}
\end{figure}

To define the last function, we fix an orientation of $\Ddot$. Then choose
one corner around 
each crossing of $\Ddot$ and associate a rational function with the corner as
in Figure~\ref{fig.omega2}. These functions are well-defined for essential
corner. The \emph{second normalizing function} $\Omega_2$ is defined as the
product of these rational functions over all essential corners among chosen
ones. 

\begin{figure}[htpb!]
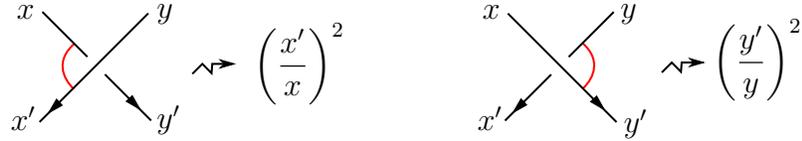
\caption{Rational functions associated with crossings.}
\label{fig.omega2}
\end{figure}

We now have all the ingredients to define the Ohtsuki--Takata invariant.

\begin{definition}[\cite{OT2015}]
The \emph{Ohtsuki--Takata invariant} of an essential open diagram $\Ddot$
is defined as
\begin{equation}
\label{eqn.OT}
\omega(\Ddot) := \pm \frac{\Omega_1 \Omega_2}{2} \, \det \left (
x_i \dfrac{\partial}{\partial x_i} \left(x_j\, \dfrac{\partial V}{ \partial x_j}
\right)\right)_{1 \leq i,j \leq n} 
\end{equation}
where the right-hand side is evaluated at the solution $x^\circ$. 
\end{definition}

It was proved that the Ohtsuki--Takata invariant is invariant under Reidemeister
II and III moves involving interior crossings of essential diagrams
\cite[Prop.4.3]{OT2015}, and was conjectured that the invariant $\omega(\Ddot)$ is
equal  to the adjoint Reidemeister torsion associated to the holonomy representation.

\noindent
{\bf Example~\ref{ex.61} continued.}
The normalizing functions $\Omega_1$ and $\Omega_2$  of the essential open diagram
$\Ddot$ given in Figure~\ref{fig.example}  are given by
\begin{equation*}
\Omega_1 =
\left(1-\frac{x_2}{x_1}\right)
\left( 1-\frac{x_3}{x_2}\right), \qquad
\Omega_2  =  \frac{1}{x_1^2} x_1^2 \, \dfrac{1}{x_3^2} \, x_3^2 = 1 \,.
\end{equation*}
The Ohtsuki--Takata invariant, evaluated at the geometric solution~\eqref{61xapp}
is $\omega(\Ddot) \approx  -0.487465 - 1.738045 i$. This agrees with
the adjoint Reidemeister torsion of the knot $6_1$.

%%%%%%%%%%%%%%%%%%%%%%%%%%%%%%%%%%%%%%%%%%%%%%%%%%%%%%%%%%%%%%%%%%%%%%%%%%%%
%%%%%%%%%%%%%%%%%%%%%%%%%%%%%%%%%%%%%%%%%%%%%%%%%%%%%%%%%%%%%%%%%%%%%%%%%%%%

\section{From planar diagrams to ideal triangulations}
\label{sec.octa}

In this section we recall a well-known connection between planar projections
of (hyperbolic) knots and their ideal triangulations. The history of the subject
predates quantum topology, and the sought algorithm 
takes as input a planar projection $D$ of a knot $K$, produces as an intermediate
stage an octahedral cell decomposition of its complement minus two points, and
eventually produces an ideal triangulation  of its complement $S^3 \setminus K$.
This algorithm due to Thurston appears in the code of \texttt{SnapPy} in the original
version due to Weeks~\cite{snappy}.
It is described in detail in~\cite{Weeks}, and forms the core  method of
\texttt{SnapPy} for computing numerically (or exactly) the complete hyperbolic
structure of a hyperbolic knot (or more generally, link) complement. 

Years after its discovery and its use in hyperbolic geometry, it was realized that
the above algorithm has applications to quantum topology, and more precisely to
the volume conjecture. The reason being obvious syntactical similarities between the
octahedral decomposition of a planar projection of a knot and the state-sum formula
for its Kashaev invariant, where the shapes of 4 tetrahedra around a crossing of the
knot match with the quantum factorials in Kashaev's $R$-matrix of his invariant.
This connection was discussed by D. Thurston in relation to the Kashaev invariant
and its asymptotics~\cite{Thu99}, and later by Yokota~\cite{Yokota02},
see also~\cite[Sec.8.2]{GL:asy}.
  
The algorithm consists of two steps
\be
\label{2steps}
\Ddot \curly (\calO_D,s) \curly \calT_{\Ddot}
\ee
which we next discuss in detail.

\subsection{Octahedral decomposition}
\label{sec.oc}

Let $K$ be a knot in $S^3$ and $D$ be a knot diagram of $K$.
Viewing $D$ as an embedded 4-valent graph  in $\BR^2$, a \emph{segment}
refers to an edge of the graph and a \emph{region} refers to a connected component
of the complement of the graph. If we add over/under-pass information at every
crossing as usual, the graph splits up into intervals, which we call \emph{over-arcs}.
Reversing all the over/under-pass information, the graph splits up into
different intervals, which we call \emph{under-arcs}.

Place an ideal octahedron at every crossing as in Figure~\ref{fig.octahedronatc}
(left) and add an edge joining the top and bottom vertices so that the
octahedron is divided into four tetrahedra. Then each tetrahedron is placed at
a corner and has two faces glued to adjacent tetrahedra. We record these
information by triangles and dashed lines as in Figure~\ref{fig.octahedronatc}
(right). Here a triangle represents a tetrahedron, viewed from the top.
We classify the edges of an octahedron into four types:
\begin{itemize}
\item
  a \emph{C-edge} is the edge  joining the top and bottom vertices; 
\item
  an \emph{R-edge} is an edge lying in the equator of the octahedron;
\item
  an \emph{O-edge} (resp., \emph{U-edge}) is a non-equatorial edge,
  transverse to the overpass (resp., underpass) when viewed from the top.
\end{itemize} 
Note that each (outer) face of an octahedron has one R-edge, one O-edge and
one U-edge.

\begin{figure}[htpb!]
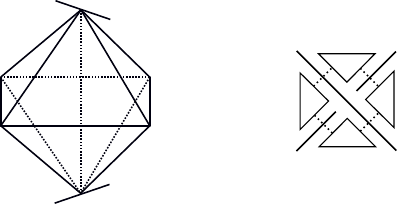
\caption{An octahedron at a crossing.}
\label{fig.octahedronatc}
\end{figure}

A segment joins two crossings, hence two octahedra. On each side of the
segment, there is a pair of faces that face each other along the segment. See
Figure~\ref{fig.octahedron} (left) for the alternating case. We glue those two
faces so that edges of the same type are identified. In this way, each segment
determines two face-pairings. We record them with dashed lines having
signs at the ends, shown as in Figure~\ref{fig.octahedron} (right). Here the
sign indicates whether the face lies in the upper or lower part of the octahedron.
Then every triangle is attached with four dashed lines: one with
$+$, another with $-$, and the others with no signs. Namely, all faces of the
tetrahedra are glued. This results in an ideal cell decomposition $\calO_D$
of $S^3 \setminus (K \cup \{p,q\})$ (with $p$ and $q$ two points not in $K$)
with cells being ideal octahedra, known as the \emph{octahedral decomposition}
associated with $D$.  We refer to \cite{KKY18} for details.

\begin{figure}[htpb!]
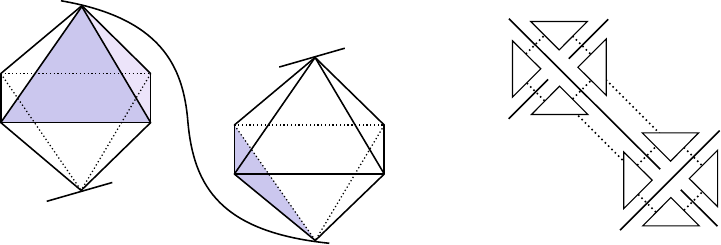
\caption{Two face-pairings along a segment.}
\label{fig.octahedron}
\end{figure}

When we glue octahedra, edges are identified only if they have the same type.
It follows that we can classify edges of $\calO_D$ into four types:
C, R, O and U-edges. In addition,
\begin{itemize}
\item a C-edge corresponds to a crossing of $D$;
\item an R-edge corresponds to a region of $D$, i.e., two R-edges
are identified if and only if they lie in the same region, viewed from the top;
\item an O-edge corresponds to an over-arc of $D$, i.e., two
O-edges are identified if and only if they are attached to the same over-arc;
\item a U-edge corresponds to an under-arc of $D$, i.e., two
U-edges are identified if and only if they are attached to the same under-arc. 
\end{itemize}
Note that if $D$ has $n$ crossings, then we have
\begin{equation*}
  \# ( \textrm{tetrahedra of } \calO_D) = 4n\, , \quad
  \# ( \textrm{edges of } \calO_D) = 4n+2
\end{equation*}
with $n$ C-edges, $(n+2)$ R-edges, $n$ O-edges and $n$
U-edges.

\subsection{Collapsing the octahedral decomposition to an ideal triangulation}  
\label{sec.collap}

In this section we discuss how to collapse the ideal cell decomposition $\calO_D$
of $S^3\setminus (K \cup \{p,q\})$ into an ideal triangulation $\calT_D$ of
$S^3\setminus K$. The collapsing depends on fixing an alternating segment $s$
of $D$. By alternating we mean that we overpass the diagram at one end
of $s$ and underpass at the other end. We denote by $r_1$ and $r_2$ two regions
adjacent to $s$, and by $o$ and $u$ the over-arc and under-arc containing
$s$, respectively (see Figure~\ref{fig.region2}).

Recall that the octahedral decomposition $\calO_D$ has four edges
corresponding to the regions $r_1$, $r_2$, the over-arc $o$, and the under-arc
$u$, respectively. We remove all tetrahedra of $\calO_D$ that contain
one of those four edges.
When viewed from the top, they are represented by triangles lying in $r_1$ or
$r_2$, and ones lying beside $o$ or $u$. After removing those triangles as well
as dashed lines in between, we add new dashed lines or arcs in an obvious way
so that every dashed line is connected, shown as in Figure~\ref{fig.region2}.
Then all faces of the remaining tetrahedra are paired along dashed lines
or arcs. This results in an collapsed ideal triangulation
of $S^3 \setminus K$.
As the pair $(D,s)$ is determined by the open diagram $\Ddot$ obtained by cutting
$D$ across $s$, we denote the collapsed triangulation simply by
$\calT_{\Ddot}$. Note that $s$ being alternating implies the technical
assumption $(\dagger)$ on $\Ddot$ in Section~\ref{sub.OT}.

\begin{figure}[htpb!]
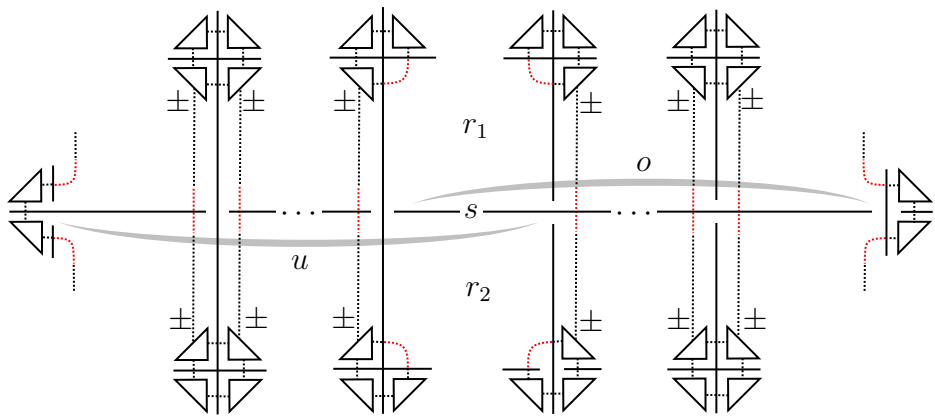
\caption{Modified face-pairings.}
\label{fig.region2}
\end{figure}

Suppose that the over-arc $o$, the under-arc $u$, and the regions $r_1$ 
and $r_2$ have $n_o$, $n_u$, $n_{r_1}$ and $n_{r_2}$ segments of
$D$, respectively. When $\calO_D$ collapses into
$\calT_{\Ddot}$, tetrahedra lying in $r_1$ or $r_2$, and ones lying beside $o$ or
$u$ disappear: total $(n_{r_1}+n_{r_2} + 4n_u+4n_o-8)$ disappear. 
Therefore, if $D$ has $n$ crossings, 
\begin{equation*}
\# ( \textrm{tetrahedra of } \calT_{\Ddot}) = 4n-n_{r_1}-n_{r_2}-4n_u-4n_o+8 \,.
\end{equation*}

On the other hand, when $\calO_D$ collapses into $\calT_{\Ddot}$,
edges of different types can be identified. In this case, we keep the bigger
type according to the following order.
\begin{equation*}
C<  R < O , U \,.
\end{equation*}
For instance, if a C-edge and an O-edge are identified, we call the identified
one an O-edge. Exceptionally, if an O-edge and a U-edge are identified, we
retain both types: the identified edge is both an O-edge and a U-edge. 

To describe how edges change explicitly, we denote the crossings of $u$
and $o$ by $c_1,\ldots,c_{n_u + 1}$ and $c_{n_u},\ldots, c_{n_u+n_o}$
respectively, in order so that the common ones, $c_{n_u}$ and
$c_{n_u+1}$, are the crossings of $s$.  We also denote a segment of $u$ or $o$ by
$[c_i,c_{i+1}]$ for $1 \leq i \leq n_o+n_u-1$. 
When $\calO_D$ collapses into $\calT_{\Ddot}$,
\begin{itemize}
\item $(n_u + n_o)$ C-edges disappear. These correspond to
  $c_1, \ldots, c_{n_u + n_o}$: more precisely, 
\subitem \textbf{(Ca)} the one corresponding to $c_1$ is identified with the
U-edge corresponding to the under-arc of $c_1$; 
\subitem \textbf{(Cb)} the one  corresponding to $c_{n_u+n_o}$ is identified
with the O-edge corresponding to the over-arc of  $c_{n_u+n_o}$; 
\subitem \textbf{(Cc)} the others disappear.  
\smallskip

\item $(n_u+n_o+2)$ R-edges disappear. Precisely, two R-edges
corresponding to regions adjacent to $[c_i,c_{i+1}]$ for 
$1 \leq i \leq n_u+n_o-1$: 
\subitem \textbf{(Ra)} disappear if $i=n_u$ (i.e., if $[c_i,c_{i+1}]=s$); 
\subitem \textbf{(Rb)} are identified with the O-edge corresponding to the
over-arc of $c_1$ if $i=1$;
\subitem \textbf{(Rc)} are identified with the U-edge corresponding to the under-arc
of $c_{n_u+n_o}$ if $i=n_u + n_o -1$;
\subitem \textbf{(Rd)} are identified to one R-edge, otherwise.  
\smallskip

\item $(n_u + n_o -3)$ O-edges disappear. Precisely,
\subitem \textbf{(Oa)} the O-edge corresponding to the over-arc containing
$[c_i, c_{i+1}]$  vanishes for $1 < i \leq n_u$; 
\subitem \textbf{(Ob)} two O-edges corresponding to adjacent over-arcs at
$c_i$ other than $o$ are identified for $n_u+1 < i  < n_{u} + n_{o}$. 
\smallskip

\item $(n_u + n_o -3)$ U-edges disappear. Precisely, 
\subitem \textbf{(Ua)} the U-edge corresponding to the under-arc containing
$[c_i, c_{i+1}]$ vanishes for $n_u < i \leq n_o+n_u$; 
\subitem \textbf{(Ub)} two U-edges corresponding to adjacent under-arcs at
$c_i$ other than $u$ are identified for $1 < i  < n_{u}$.
\end{itemize}
In addition, for each segment $s' \neq s$ lying in $r_1$ or $r_2$, an O-edge
corresponding to the over-arc containing $s'$ is identified with a U-edge
corresponding to the under-arc containing $s'$. This leads to
$(n_{r_1}+n_{r_2}-2)$ identifications between O-edges and U-edges. 
We refer to Appendix~\ref{sec.app} for details.  Therefore, if $D$ has $n$
crossings, 
\begin{align*}
\# ( \textrm{edges of } \calT_{\Ddot}) &= (4n + 2) - (n_u + n_o) - (n_u + n_o +2) \\
 & \quad \quad - (n_u + n _o -3) - (n_u + n _o -3) - (n_{r_1} + n_{r_2} -2 ) \\
&=  \, 4n  - n_{r_1} - n_{r_2}-4n_o -4n_u + 8 \, .
\end{align*}
This double-checks that the number of edges is equal to the number of tetrahedra.

\noindent
{\bf Example~\ref{ex.61} continued.} For the open diagram $\Ddot$ of
Figure~\ref{fig.example}, its closure $D$ is shown on the left of
Figure~\ref{fig.examplecc}. The diagram $D$ has six crossings, and thus the octahedral
decomposition $\calO_D$ has
six C-edges and eight R-edges, as well as six O-edges $o_1,\ldots,o_6$ and six
U-edges $u_1,\ldots , u_6$. It is convenient to label a segment of $D$
by $o_i/u_j$ where $o_i$ and $u_j$ correspond to the over and under-arcs containing
the segment, respectively.

We choose an alternating segment $s$ as the one with $o_6/u_1$ (the leftmost one).
Then, when $\calO_D$ collapses into the collapsed triangulation $\calT_{\Ddot}$,
\begin{itemize}
\item
  four C-edges adjacent to the segment with $o_1/u_1$, $o_6/u_1$, or $o_6/u_6$ disappear,
\item
  two R-edges adjacent to the segment with $o_6/u_1$ disappear,
\item
  two R-edges adjacent to the segment with $o_1/u_1$ are identified with $o_1$,
\item
  two R-edges adjacent to the segment with $o_6/u_6$ are identified with $u_6$,
\item $o_6$ and $u_1$ disappear.
\end{itemize} 
In addition, we obtain identifications
\begin{equation*}
  o_4 = u_4, \quad o_1 = u_2, \quad  o_2=u_2, \quad  o_5 = u_6,
  \quad o_3 = u_3, \quad o_4=u_5
\end{equation*}
from segments lying in the unbounded region of $D$.
It follows that the collapsed triangulation $\calT_{\Ddot}$ has two C-edges,
two R-edges, and four O or U-edges:
\begin{equation*}
\{(o_1 = o_2 = u_2), \ (o_3 = u_3), \ (o_4=u_4=u_5), \ (o_5 = u_6) \} \,.
\end{equation*}

To connect the 8 shape parameters of the collapsed triangulation $\calT_{\Ddot}$
with the 3 variables of the Yokota potential of Figure~\ref{fig.example}, we will
eliminate five of them by solving to the gluing equations of two C-edges  and two
R-edges, and completeness equation of the meridian; see Section~\ref{sec.reduction}
for details.

\begin{figure}[htpb!]
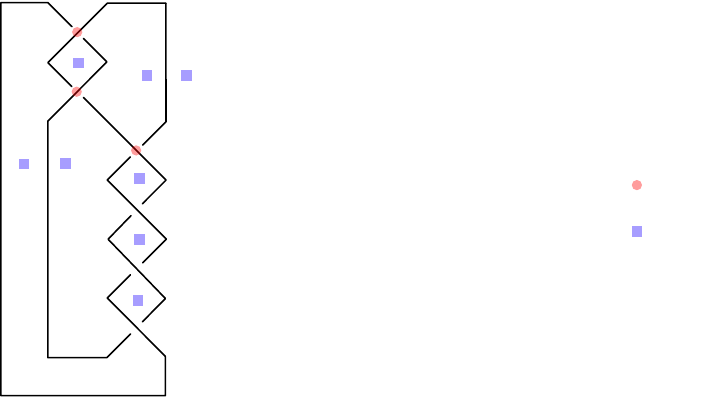
\caption{Edges of $\calO_D$ and $\calT_{\Ddot}$.}
\label{fig.examplecc}
\end{figure}

%%%%%%%%%%%%%%%%%%%%%%%%%%%%%%%%%%%%%%%%%%%%%%%%%%%%%%%%%%%%%%%%%%%%%%%%%%%%
%%%%%%%%%%%%%%%%%%%%%%%%%%%%%%%%%%%%%%%%%%%%%%%%%%%%%%%%%%%%%%%%%%%%%%%%%%%%

\section{Proof of the main theorem}
\label{sec.pf}

In this section we prove Theorem~\ref{thm.main2} below, which combined with the work
of Ohtsuki--Takata~\cite{OT2015} implies Theorem~\ref{thm.main}.
We fix an open diagram $\Ddot$ of a hyperbolic 2-bridge knot $K$ as in the
following figure, where boxes are twist regions.

\begin{figure}[htpb!]
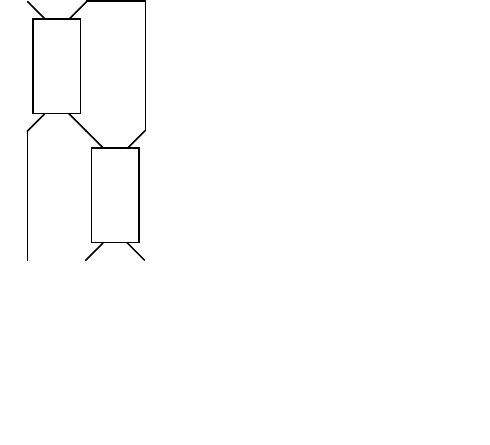
\caption{An open diagram of a 2-bridge knot.}
\label{fig.planardiagram}
\end{figure}

\begin{theorem}
\label{thm.main2}
For $\Ddot$ as above, $\tau(\calT_{\Ddot})=\pm\omega(\Ddot)$. 
\end{theorem}

\subsection{Flattening of the collapsed triangulation}
\label{sec.flat}

In this subsection, we find an explicit flattening ${\boldsymbol f}=(f,f',f'')$
of $\calT_{\Ddot}$. Recall from Section~\ref{sec.OL} that ${\boldsymbol f}$ consists
of integer triples ${\boldsymbol f}_j=(f_j,f'_j,f''_j) \in \BZ^3$, one for each
tetrahedron $\Delta_j$ of $\calT_{\Ddot}$ and is required to satisfy
$f_j + f'_j +f''_j =1$ for each $\Delta_j$ and
\begin{equation}
\label{eqn.flat}
e_i : \quad \sum_{j} G_{ij} f_j +  G'_{ij} f'_j  + G''_{ij} f''_j  =2
\end{equation}
for each edge $e_i$. We denote by $\langle {\boldsymbol f}, e_i \rangle$
the left-hand side of Equation~\eqref{eqn.flat}.

We fix a quad type of each $\Delta_j$ so that $f_j$ is assigned to the vertical
and horizontal edges of $\Delta_j$ (the edges $\overline{0\infty}$ and
$\overline{xy}$ in Figure~\ref{fig.crossing2}).
Denoting $\Delta_j$ with a corner of $\Ddot$, we choose an integer triple
${\boldsymbol f}_j$ as follows.
\begin{equation}
\label{eqn.fl1}
{\boldsymbol f}_j  =
\begin{cases}
(1,0,0) & \textrm{for the N and S-corners of a crossing} \\
(0,1,0) & \textrm{for the E and W-corners of a positive crossing} \\
(0,0,1) & \textrm{for the E and W-corners of a negative crossing} \, .
\end{cases}
\end{equation}
Here N, E, S or W indicates the position of a corner with respect to the nearby
crossing. Note that each tetrahedron has two edges assigned with 1 and four edges
assigned with 0. When we compute  $\langle {\boldsymbol f}, e \rangle$, the ones
with 0 have no influence, and thus it suffices to consider the ones with $1$, which
are marked by red in Figure~\ref{fig.octahedronatd}.

\begin{figure}[htpb!]
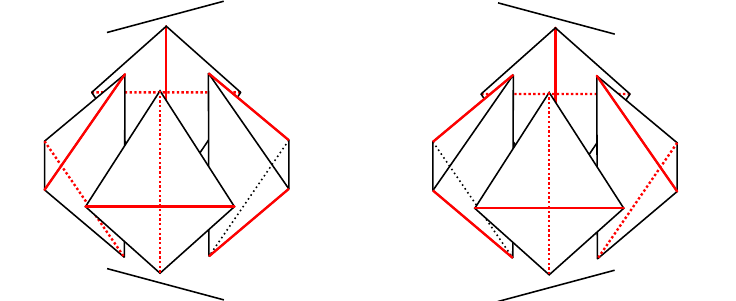
\caption{A choice of ${\boldsymbol f}$.}
\label{fig.octahedronatd}
\end{figure}

Recall from Section~\ref{sec.octa} that $\calT_{\Ddot}$ has four edge-types: C, R, O
and U-edges. We first compute $\langle {\boldsymbol f}, e \rangle$ for C and R-edges.  
As these edges recieve contributions only from $f_j$'s (not from $f'_j$'s and
$f''_j$'s),  only N and S-corners are involved in the computation.
For simplicity we call a segment with $\kappa \in \{ 0,1,\infty\}$ assigned a
$\kappa$-segment.

\begin{lemma} 
\label{lem.C}
We have $\langle {\boldsymbol f}, e \rangle=2$ for all C-edges $e$.
\end{lemma}

\begin{proof}
Recall that a C-edge corresponds to a crossing with no adjacent $0$ and
$\infty$-segments. Thus a corner around it is essential, unless adjacent segments
are both 1-segments. Given that $\Ddot$ is given as in Figure~\ref{fig.planardiagram},
both the N and S-corners are essential and contribute 1 to
$\langle {\boldsymbol f}, e \rangle$. This proves that
$\langle {\boldsymbol f}, e \rangle=2$. 
\end{proof}

\begin{lemma} 
\label{lem.R}
We have $\langle {\boldsymbol f}, e \rangle=2$ for all R-edges $e$.
\end{lemma}

\begin{proof}
Let us consider a region of $\Ddot$ that corresponds to an R-edge $e$. The S-corner of
the topmost crossing contributes 1 to  $\langle {\boldsymbol f}, e \rangle$. Similarly,
the N-corner of the bottommost crossing contributes 1 to
$\langle {\boldsymbol f}, e \rangle$. The other crossings that are neither top nor
bottom ones have no contribution; see Figure~\ref{fig.octahedronatd}. This proves that
$\langle {\boldsymbol f}, e \rangle=2$. 
\end{proof}

We now compute $\langle \boldsymbol f  , e \rangle$ for O and U-edges. Recall that O
and U-edges correspond to over and under-arcs of $\Ddot$, respectively. In what follows,
for simplicity we confuse an over or under-arc of $\Ddot$ with the corresponding O or
U-edge of $\calT_{\Ddot}$. In addition, we assume that starting from the top endpoint of
$\Ddot$, we underpass the first crossing (thus from the bottom endpoint, we overpass the
first crossing). The opposite case can be proved similarly by exchanging the roles of
$\infty$ and  $0$. 

The diagram $\Ddot$ has one cap, and we choose one crossing $c_\infty$ and one segment
$s_\infty$ around it shown as in Figure~\ref{fig.infty}. Similarly, we choose one
crossing $c_0$ and one segment $s_0$ around the cup of $\Ddot$ as in
Figure~\ref{fig.infty}. We denote by $o_\kappa$ and $u_\kappa$ for $\kappa = 0,\infty$
the over and under-arcs containing the segment $s_\kappa$, respectively.

\begin{figure}[htpb!]
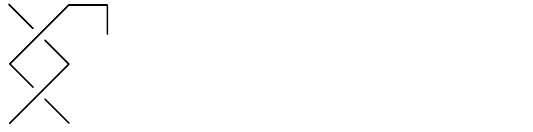
\caption{Cap and cup of $\Ddot$.}
\label{fig.infty}
\end{figure}

\begin{lemma} 
\label{lem.O}
For an O or U-edge $e$, we have
\begin{equation*}
\langle \boldsymbol f, e \rangle
=\begin{cases}
3 & \textrm{if } e =  u_\infty \textrm{ or } o_0 \,,\\
1 & \textrm{if } e = o_\infty \textrm{ or }  u_0 \,,\\	
2 & \textrm{otherwise} \,.
\end{cases}
\end{equation*}
\end{lemma}

\begin{proof}
Following the diagram from the top endpoint, we encounter all over and under-arcs.
Thus for each O and U-edge we can think of a crossing where it starts or ends.
Except for the four edges $u_0,u_\infty, o_0$ and $o_\infty$, every O and U-edge
$e$, $\langle \boldsymbol f \, , e \rangle$ receives one from the crossing where it
starts and another one from where it ends; see Figure~\ref{fig.octahedronatd}. This
proves that $\langle \boldsymbol f \, , e \rangle=2$. Note that this includes the
case when an O-edge and a U-edge are identified. Such identification happens when
they meet at the W or E-corner lying in the unbounded region, and the fact that we
do not have a tetarhedron at that corner preserves the fact
$\langle \boldsymbol f \, , e \rangle=2$.

For $e = u_\infty$,  it receives one from the crossing where it starts and another
one from where it ends, as before. However, as the face-gluing around $c_\infty$ is
different from other crossings, i.e., as the vertical edge at $c_\infty$ is identified
with $u_\infty$, it receives one additional contribution from $c_\infty$.

It follows that $\langle \boldsymbol f \, , e \rangle=3$. We deduce the same for
$e=u_0$ similarly.
	
For $e = o_\infty$, $\infty$-segments make corners around the crossing where  it starts
not essential. This results in a tetrahedron being missing at the E or W-corner of its
start crossing, and thus $ \langle \boldsymbol f \, , e \rangle =1 $. We deduce the
same for $e =u_0$ similarly.
\end{proof}

\begin{lemma}
\label{lem.overunder}
Let $e_i$ and $e_j$ be two edges of $\calT_{\Ddot}$ that correspond to consecutive under
or over-arcs. Then there is a triple $\boldsymbol g=(g,g', g'')$ of vectors such
that $g_j + g'_j + g_j''=0$ and
\begin{equation*}
\langle {\boldsymbol f + \boldsymbol g} ,  e \rangle -\langle {\boldsymbol f}, e \rangle
=\begin{cases}
-1 & \textrm{if } e = e_i \\
1 & \textrm{if } e = e_j \\
0 & \textrm{otherwise}\, .
\end{cases}
\end{equation*}
In addition, such $\boldsymbol g$ can be chosen independently of $\boldsymbol f$ .
\end{lemma}

\begin{proof}
Let $a_1$ and $a_2$ be the over or under-arcs corresponding to the edges $e_i$ and
$e_j$, respectively. As they are adjacent, there is a common crossing. If $a_1$ and
$a_2$ are under-arcs,  then we define $\boldsymbol g$ by assigning integer triples to
the corners at the common crossing as in Figure~\ref{fig.crossing4} (left); if $a_1$
and $a_2$ are over-arcs, then as in Figure~\ref{fig.crossing4} (right). Then one can
easily check that $\boldsymbol g$ satisfies the desired condition.

\begin{figure}[htpb!]
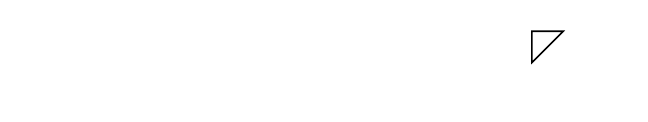
\caption{A choice of $\boldsymbol g$.}
\label{fig.crossing4}
\end{figure}
\end{proof}

\begin{remark}
\label{rmk.ratio}
In the proof of Lemma~\ref{lem.overunder}, we used two tetrahedra that are below
the arcs $a_1$ and $a_2$, but one may use two tetrahedra that are above the arcs. This
implies that a choice of $\boldsymbol g$ is not unique. However, regardless of a
choice of $\boldsymbol g$, the product 
\begin{equation*}
{\boldsymbol \zeta}^{\boldsymbol g}:=
\prod_{j}\zeta_j^{g_j} {\zeta_j'}^{g_j'} {\zeta_j''}^{g_j''}
\end{equation*}
is invariant: if there are two choices $\boldsymbol g_1$ and $\boldsymbol g_2$, then
for any flattening $\boldsymbol f$ of $\calT_{\Ddot}$,
$\boldsymbol f + \boldsymbol g_1 - \boldsymbol g_2$ is also a flattening, and thus
\begin{equation*}
  {\boldsymbol \zeta}^{\boldsymbol f + \boldsymbol g_1 -\boldsymbol g_2} =
  {\boldsymbol \zeta}^{\boldsymbol f}  \quad \Rightarrow \quad
  {\boldsymbol \zeta}^{\boldsymbol g_1 }= {\boldsymbol \zeta}^{\boldsymbol g_2} \,.
\end{equation*}
Explicitly, ${\boldsymbol \zeta}^{\boldsymbol g}$ agrees with $z_1z_2$ for
Figure~\ref{fig.crossing3}~(left) and with $1/(z_1z_2)$ for Figure
~\ref{fig.crossing3}~(right) where $z_1$ and $z_2$ are shape parameters placed at
the corners shown as in Figure~\ref{fig.crossing3}.
\end{remark}

Roughly speaking, Lemma ~\ref{lem.overunder} allows us to take one contribution
from an arc, either over or under, and give it to the next arc, and thus to any arc
by applying the lemma multiple times. We take one contribution from $o_0$ (resp.,
$u_\infty$) and give it to $o_\infty$ (resp., $u_0$). This results in a triple
$\boldsymbol g$ such that (see Lemma~\ref{lem.O})
\begin{equation*}
\langle \boldsymbol f+\boldsymbol g , e \rangle  = 2
\end{equation*}
for all edges $e$. Namely, $\boldsymbol f+\boldsymbol g $ is a flattening of
$\calT_{\Ddot}$.

\subsection{Variable reduction}
\label{sec.reduction}

In this subsection, we reduce the number of shape parameters by solving the gluing
equations for the C and R-edges, and the completeness equation. The following lemma
will be used repeatedly.

\begin{lemma}
\label{lem.red}
Let $f_1,\ldots,f_k$ be functions in variables $x_1,\ldots,x_k$ and fix a solution to 
\begin{equation*}
\exp(f_1) = \cdots = \exp(f_k)=1\,.
\end{equation*} Suppose that $\exp(f_k)$ is of the form $g(x_1,\ldots,x_{k-1}) \, x_k$
so that we can substitute $x_k$ with $1/g(x_1,\ldots,x_{k-1})$ by solving $\exp (f_k)=1$.
Then we have
\begin{equation}
\label{eqn.xx}
  \det \left ( \dfrac{\partial (f_1,\ldots,f_k)}{\partial(x_1,\ldots,x_k)} \right)
  = \dfrac{1}{x_k}
\det \left ( \dfrac{\partial (f'_1,\ldots,f'_{k-1})}{\partial(x_1,\ldots,x_{k-1})} \right)
\end{equation}
where $f'_i$ is obtained from $f_i$ by  substituting $x_k$ with $1/g(x_1,\ldots,x_{k-1})$.
\end{lemma}
      
\begin{proof}
It follows from $\exp(f_k) = g(x_1,\ldots,x_{k-1}) \, x_k$ that the $(k,j)$-entry of
the Jacobian matrix $\partial (f_1,\ldots,f_k)/\partial(x_1,\ldots,x_k)$ is $1/x_k$
for $j=k$, and  $g^{-1} \,\partial g / \partial x_j$ for $1 \leq j < k$. Keeping the
last row, we make the $i$-row for $i \neq k$ have the last entry 0 by employing
elementary row operations. Then  the $(i,j)$-entry becomes 
\begin{align*}
  \frac{\partial f_i}{\partial x_j}- \frac{x_k}{g} \frac{\partial g}{\partial x_j}
  \frac{\partial f_{i}}{\partial x_k} & = \frac{\partial f_i}{\partial x_j}
  - \frac{1}{g^2} \frac{\partial g}{\partial x_j} \frac{\partial f_{i}}{\partial x_k} \\
  &=\frac{\partial f_i}{\partial x_j} + \frac{\partial g^{-1}}{\partial x_j}
  \frac{\partial f_{i}}{\partial x_k} \\
  &=\frac{\partial f_i}{\partial x_j} + \frac{\partial x_k}{\partial x_j}
  \frac{\partial f_{i}}{\partial x_k}\,.
\end{align*}
As the last term is equal to $\partial f'_i/\partial x_j$, we deduce
Equation~\eqref{eqn.xx}. 
\end{proof}

Recall that the gluing equation of an edge $e_i$ of $\calT_{\Ddot}$ is of the form
\begin{equation*}
\sum_{j=1}^N  G_{ij} \log z_j + G'_{ij} \log z'_j
+ G''_{ij} \log z''_j  =2 \pi i \, .
\end{equation*}
We denote by $g_i$ the left-hand side so that $\exp(g_i)=1$.
By rearranging the index of edges, we may assume that  the first $n+1$ equations
$g_1,\ldots, g_{n+1}$ are for O and U-edges and the rest, say $g_{n+2},\ldots, g_N$,
are for C and R-edges. By replacing the $(n+1)$-st equation $g_{n+1}$ with $g_\mu$, the
logarithm form of the completeness equation of a meridian, the 1-loop invariant is
given by (see Remark~\ref{rmk.der})
\begin{equation}
\label{eqn.t1}
\tau(\calT_{\Ddot}) = \pm \frac{1}{2{\boldsymbol \zeta}^{\boldsymbol f + \boldsymbol g}}
\det \left(\dfrac{\partial (g_1,\ldots, g_{n}, g_\mu,g_{n+2},\ldots,g_N)}
{\partial (z_1,\ldots,z_N)} \right) 
\end{equation}
where $\boldsymbol f + \boldsymbol g$ is the flattening of $\calT_{\Ddot}$ chosen in
Section~\ref{sec.flat}.

In what follows,  we solve all C and R-edge equations $g_{n+2},\ldots,g_{N}$ as well
as the completeness equation $g_\mu$. This eliminates $N-n$ shape parameters, and
we may assume that those are  $z_{n+1},\ldots, z_{N}$ by rearranging the index of
shape parameters.
\begin{itemize}
\item
  Recall that a C-edge corresponds to a crossing of $\Ddot$. Given that $\Ddot$ is
  given as in Figure~\ref{fig.planardiagram}, one of E and W-corners at the crossing
  is non-essential, and thus the gluing equation of the C-edge is either
  \begin{equation*}
  z_{N} z_{E} z_{S}  = 1 \quad \textrm{or} \quad z_{N} z_{W} z_{S} =1 \, .
\end{equation*}
where $z_X$ means the shape parameter at the X-corner. Employing Lemma~\ref{lem.red},
we eliminate the shape parameter $z_S$ of the S-corner.
\item
  Similarly, an R-edge  corresponds to a region of $\Ddot$, and its gluing equation is
  of the form $ z_{i_1} \cdots  z_{i_k}  = 1$ where the S-corner of the topmost
  crossing and the N-corner of the bottommost crossing contribute one shape parameter.
  Using Lemma~\ref{lem.red}, we eliminate the shape parameter at the N-corner of the
  bottommost crossing. 
\item
  Lastly, the completeness equation for a meridian is given as
  $z_{i_1}^{\pm1} z_{i_2}=1$ where $z_{i_1}$ and $z_{i_2}$ are shape parameters placed
  at corners shown as in Figure~\ref{fig.meridian}. We eliminate the shape parameter
  $z_{i_2}$ by employing Lemma~\ref{lem.red}.
\begin{figure}[htpb!]
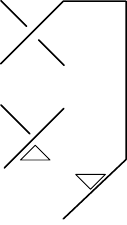
\caption{A completeness equation for a meridian.}
\label{fig.meridian}
\end{figure}
\end{itemize}

As mentioned earlier, the above variable reduction eliminates $N-n$ shape parameters
$z_{n+1},\ldots,z_N$, and Lemma~\ref{lem.red} implies that
\begin{equation}
\label{eqn.t2}
\det \left(\dfrac{\partial (g_1,\ldots, g_{n}, g_\mu,g_{n+2},\ldots,g_N)}
  {\partial (z_1,\ldots,z_N)} \right)  = \frac{1}{z_{n+1} \cdots z_N} \det
\left(\dfrac{\partial (g'_1,\ldots, g'_{n})}{\partial (z_1,\ldots,z_{n})} \right) 
\end{equation}
where $g'_i$ is obtained from $g_i$ by eliminating the $N-n$ shape parameters
$z_{n+1},\ldots,z_N$. A simple calculation shows that the number $n$ of the remaining
variables is equal to the number of variables that are assigned to the open diagram
$\Ddot$. Namely, the matrix in the right-hand side of Equation~\eqref{eqn.t2} has
the same size as the matrix in~\eqref{eqn.OT}.

\begin{lemma} 
\label{lem.cc}
We have 
\begin{equation*}
\det \left ( x_i \dfrac{\partial}{\partial x_i} \left(x_j\, \dfrac{\partial V}{
\partial x_j} \right)\right)_{1 \leq i,j \leq n} = \pm z_1 \ldots z_{n} \det
\left(\dfrac{\partial (g'_1,\ldots, g'_{n})}{\partial (z_1,\ldots,z_{n})} \right) 
\end{equation*}
where $V$ is the potential function of $\Ddot$.
\end{lemma}

\begin{proof}
It is proved in \cite[Theorem~4.1]{MY2018} that the system of equations
$x_i \partial V / \partial x_i$ is equivalent to the system of equations $g'_i$.
Hence we have
\begin{align*}
\det \left (\dfrac{\partial}{\partial x_i} \left(x_j\,
\dfrac{\partial V}{ \partial x_j} \right)\right) &= \pm \det
\left(\dfrac{\partial (g'_1,\ldots, g'_{n})}{\partial (z_1,\ldots,z_{n})} \right)
\det \left(\dfrac{\partial (z_1,\ldots, z_{n})}{\partial (x_1,\ldots,x_{n})} \right) \,.
\end{align*}
If we label the indices of $z_i$ and $x_i$ from the top of the diagram $\Ddot$ to the
bottom, we have $z_1 =x_1$ and $z_i = x_{i}/x_{i-1}$ or $x_{i-1}/x_{i}$ for $i \geq 2$.
A simple induction argument shows that
\begin{equation*}
\det \left(\dfrac{\partial (z_1,\ldots, z_{n})}{\partial (x_1,\ldots,x_{n})} \right)
= \pm  \frac{z_1 \cdots z_n}{x_1 \cdots x_n} \,.
\end{equation*}
Combining the above two equations, we conclude the lemma.
\end{proof}

\begin{lemma} We have
\label{lem.ccc}
\begin{equation*}
  \Omega_1 \Omega_2 =
  \pm \frac{\zeta_1 \cdots \zeta_N}{{\boldsymbol \zeta}^{\boldsymbol f + \boldsymbol g}}
\end{equation*}
where $\Omega_1$ and $\Omega_2$ are the normalizing functions of $\Ddot$.
\end{lemma}

\begin{proof}
Recall that the second normalizing function $\Omega_2$ depends on a choice of the
orientation of $\Ddot$. If we reverse the orientation of $\Ddot$, the contribution of
the left (resp., right) crossing in Figure~\ref{fig.omega2} to $\Omega_2$ becomes
$(y/y')^2$ (resp., $(x/x')^2$).
Considering both orientations of $\Ddot$, one can define $\Omega_2$ without an
orientation choice: $\Omega_2$ is defined as the product of $x'y/xy'$ for the first
crossing in Figure~\ref{fig.omega1} and $xy'/x'y$ for the second crossing. It follows
that $\Omega_1 \Omega_2$ is given by the product of $(\frac{x'}{x}-1)$ and
$(\frac{y}{y'}-1)$ for the first crossing, and $(\frac{x}{x'}-1)$ and
$(\frac{y'}{y}-1)$ for the second crossing.
From Figure~\ref{fig.octahedronatd} we easily computes that this product is equal
to $1/{\boldsymbol \zeta}^{\boldsymbol f}$. This proves that $\Omega_1 \Omega_2
=1/{\boldsymbol \zeta}^{\boldsymbol f}$.

On the other hand, a product of any three zetas around one crossing is 1. Given that
$\Ddot$ is given as in Figure~\ref{fig.planardiagram}, we have 
\begin{equation*}
\zeta_1 \cdots \zeta_N = \zeta_1 \zeta_N = \frac{1}{z_1 z_N}
\end{equation*}
where $z_1$ and $z_N$ are the shape parameters that appear first from the top and
bottom, respectively, of the diagram. From Remark~\ref{rmk.ratio}, we deduce that
$z_1 z_N$ is equal ${\boldsymbol \zeta}^{\boldsymbol g}$. This completes the proof.
\end{proof}

\begin{remark}
We expect that a similar lemma to the one above holds for all essential diagrams
obtained from braid closures of knots; see Appendix~\ref{sec.appb}.
\end{remark}

\begin{proof}[Proof of Theorem~\ref{thm.main2}]
Combining Equations~\eqref{eqn.t1} and \eqref{eqn.t2} together with Lemmas
~\ref{lem.cc} and~\ref{lem.ccc}, we obtain
\begin{align*}
\tau(\calT_{\Ddot}) &= \pm \frac{1}{2{\boldsymbol \zeta}^{\boldsymbol f + \boldsymbol g}}
\det \left(\dfrac{\partial (g_1,\ldots, g_{n}, g_\mu,g_{n+2},\ldots,g_N)}
{\partial (z_1,\ldots,z_N)} \right) \\
&= \pm \frac{\zeta_{n+1} \cdots \zeta_N}{2{\boldsymbol \zeta}^{\boldsymbol f + \boldsymbol g}}
\det \left(\dfrac{\partial (g'_1,\ldots, g'_{n})}{\partial (z_1,\ldots,z_{n})} \right) \\
&=\pm \frac{\zeta_1 \cdots \zeta_N}{2{\boldsymbol \zeta}^{\boldsymbol f + \boldsymbol g}}
\det \left ( x_i \dfrac{\partial}{\partial x_i} \left(x_j\, \dfrac{\partial V}{
\partial x_j} \right)\right)  \\
&=\pm \frac{\Omega_1 \Omega_2}{2} 	\det \left ( x_i \dfrac{\partial}{\partial x_i}
\left(x_j\, \dfrac{\partial V}{ \partial x_j} \right)\right) = \omega(\Ddot)\,,
\end{align*}
which concludes the proof of Theorem~\ref{thm.main2}.
\end{proof}

%%%%%%%%%%%%%%%%%%%%%%%%%%%%%%%%%%%%%%%%%%%%%%%%%%%%%%%%%%%%%%%%%%%%%%%%%%%% 
%%%%%%%%%%%%%%%%%%%%%%%%%%%%%%%%%%%%%%%%%%%%%%%%%%%%%%%%%%%%%%%%%%%%%%%%%%%%

\appendix

\section{Further discussion on collapsing}
\label{sec.app}

In this section we illustrate the collasping process explained in
Section~\ref{sec.collap}. For simplicity, we only consider alternating diagrams, but
the non-alternating case can be described in a similar way.

Let $D$ be an alternating diagram of a knot $K \subset S^3$ and  $\Ddot$ be
an open diagram obtained from $D$ by cutting a segment $s$ of $D$.
Recall from Section~\ref{sec.octa} that $\calO_D$  is an ideal
triangulation of $S^3 \setminus (K \cup \{ \pm \infty\})$ with tetrahedra placed at
corners of $D$ and that $\calT_{\Ddot}$ is an ideal triangulation obtained from
$\calO_D$ by removing some tetrahedra lying around $s$ and changing some
face-pairings.

Chopping off the vertices of the tetrahedra, we obtain a compact 3-manifold whose
boundary consists of one triangulated torus $\nu(K)$ and two triangulated spheres
$\nu(\pm \infty)$.
When $\calO_D$ collapses to $\calT_{\Ddot}$, the boundary surfaces $\nu(K)$ and
$\nu(\pm \infty)$ are glued after losing some triangles around $s$. In what follows,
we describe how these surfaces change.

\begin{itemize}
\item
  For the boundary torus $\nu(K)$, three cylinders around $s$ are removed, shown as
  in Figure~\ref{fig.region}~(right). There are six boundary circles, but the
  face-pairing results in four of them being paired; see dashed arrows in Figure
  ~\ref{fig.region}~(right). The resulting surface $\nu'(K)$ is  a cylinder that
  wraps around $K$ except for the segment $s$. 
\item
  For the triangulated sphere $\nu(+\infty)$, the removed triangles are described in
  Figure~\ref{fig.drawing}~(left). We divide the removed area into three sections: the
  body, which is the union of two regions of $D$ adjacent to $s$; the tail,
  on the left  side of the body; and the head, on the right side of the body.
  After the gluing, the tail is closed up, and both the head and body become discs; see
  dashed arrows in Figure~\ref{fig.drawing}~(left).  Namely, the sphere $\nu(+\infty)$
  becomes a sphere $\nu'(+\infty)$ with two holes.
\item
  Similarly, the other sphere $\nu(-\infty)$ also becomes a sphere  $\nu'(-\infty)$
  with two holes; see Figure~\ref{fig.drawing}~(right).
  Note that topologically, both $\nu'(\pm \infty)$ are cylinders.
\end{itemize}

\begin{figure}[htpb!]
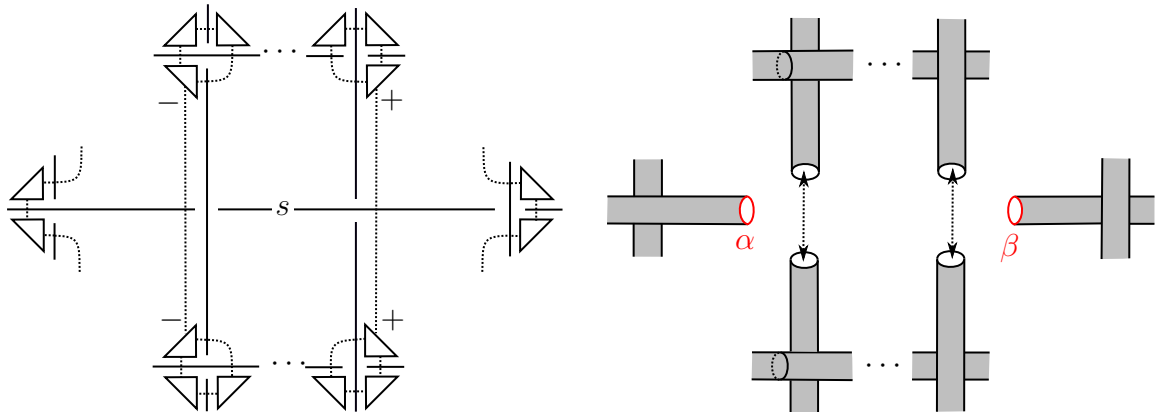
\caption{Modified face-pairings and $\nu'(K)$}
\label{fig.region}
\end{figure}

\begin{figure}[htpb!]
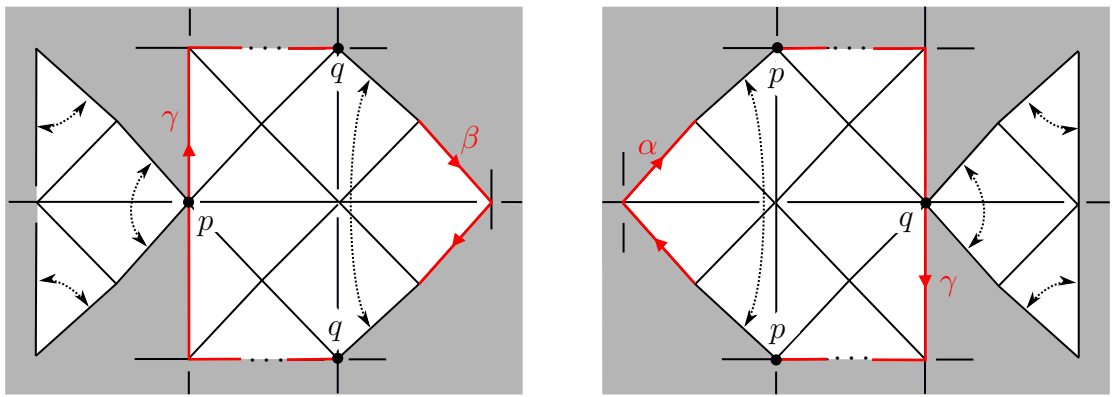
\caption{$\nu'(+\infty$) and $\nu'(-\infty)$.}
\label{fig.drawing}
\end{figure}

Two cylinders $\nu'(\pm \infty)$ are glued along the boundary circles coming from
their body parts and form one cylinder. This leads to identifications between U-edges
and O-edges. More precisely, for each segment $s'$ lying in the body part, the U-edge
corresponding to the under-arc containing $s'$ is identified with the O-edge
corresponding to the over-arc containing $s'$. The two remaining boundary circles,
coming from heads, are glued to the boundary circles of $\nu'(K)$. Given that $+\infty$
and $-\infty$ are points far above and below the diagram, respectively, the union of
$\nu'(K)$ and $\nu'(\pm \infty)$ form a tubular neighborhood of the knot $K$. This
shows that the underlying space of $\calT_{\Ddot}$ is the knot complement
$S^3 \setminus K$.

\section{Braid closures}
\label{sec.appb}

In this section we explain how to find a flattening of the collapsed triangulation
$\calT_{\Ddot}$ when $\Ddot$ is given by a closure of a braid, more precisely,
when $\Ddot$ is an open diagram obtained from a braid by taking the closure of all
strands except the first one. See Figure~\ref{fig.example2} for an example.

\begin{figure}[htpb!]
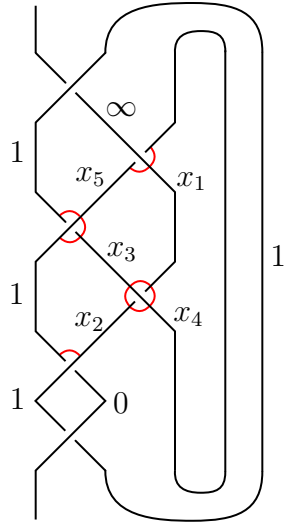
\caption{An open diagram of the knot $6_1$.}
\label{fig.example2}
\end{figure}
 
The triple $\boldsymbol{f}$ given in Equation~\eqref{eqn.fl1} still satisfies
Lemma~\ref{lem.C}, but not Lemma~\ref{lem.R}. Precisely, if we repeat the proof of
Lemma~\ref{lem.R}, we obtain the following: for an R-edge $e$, we have
$\langle {\boldsymbol f}, e \rangle=0$ if $e$ corresponds to the innermost region,
created by closing up the last strand of the braid; otherwise,
$\langle {\boldsymbol f}, e \rangle=2$. See, for instance, Figure~\ref{fig.example2}
that the innermost region does not have a crossing whose N or S-corner lies in
the region. 

To make $\langle {\boldsymbol f}, e \rangle=2$ hold for all R-edges $e$, we use the
following lemma.

\begin{lemma}
\label{lem.region}
Let $e_i$ and $e_j$ be two edges of $\calT_{\Ddot}$ that correspond to  adjacent regions.
Then there is a triple $\boldsymbol h=(h,h', h'')$ of vectors such that
$h_j + h'_j + h_j''=0$ and 
\begin{equation*}
\langle {\boldsymbol f + \boldsymbol h}, e \rangle -\langle {\boldsymbol f}, e \rangle
=\begin{cases}
-1 & \textrm{if } e = e_i \\
1 & \textrm{if } e = e_j \\
0 & \textrm{otherwise}
\end{cases}
\end{equation*}
In addition, such $\boldsymbol h$ can be chosen independently of $\boldsymbol f$.
\end{lemma}

\begin{proof}
Let $r_i$ and $r_j$ be the regions corresponding to the edges $e_i$ and $e_j$,
respectively. As they are adjacent, we can find two adjacent corners, with one lying
in $r_i$ and the other in $r_j$. If these corner are separated by an underpass, then
we define $\boldsymbol h$ by assigning integer triples to the corners as in
Figure~\ref{fig.crossing3} (left); if by an overpass, as in Figure~\ref{fig.crossing3}
(right). 
\begin{figure}[htpb!]
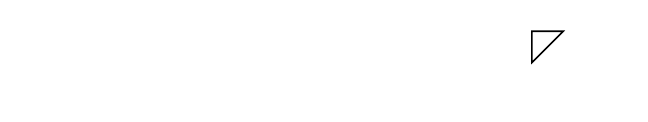
\caption{A choice of $\boldsymbol h$.}
\label{fig.crossing3}
\end{figure}
	
\noindent
Then from the first coordinate of triples, we deduce that
\begin{equation*}
  \langle {\boldsymbol f + \boldsymbol h}, e_i\rangle = \langle {\boldsymbol f},
  e_i \rangle -1, \quad  \quad  \langle {\boldsymbol f + \boldsymbol h}, e_j \rangle
  = \langle {\boldsymbol f}, e_j \rangle +1 \, .
\end{equation*}
The change of the second and third coordinates affects two edges, one U-edge and one
O-edge, but both receive one $+1$ and one $-1$, resulting in a total of 0.
\end{proof}

\begin{remark}
Adjacent regions have a segment in common, and the proof of Lemma~\ref{lem.region}
works for any pair of corners at one end of the segment. In particular,  a choice of
$\boldsymbol h$ may not be not unique. However, regardless of a choice of
$\boldsymbol h$, the product ${\boldsymbol \zeta}^{\boldsymbol h}$ is invariant (see
Remark~\ref{rmk.ratio}).
Moreover, ${\boldsymbol \zeta}^{\boldsymbol h}$ for Figure~\ref{fig.crossing3} agrees
with the inverse of what Ohtsuki and Takata referred to as $\alpha$ in \cite{OT2015}
that is assigned to the segment.
\end{remark}

Roughly speaking, Lemma ~\ref{lem.region} allows us to take one contribution from a
region and give it to an adjacent region, and thus to any region, by applying the
lemma multiple times.
We take one contribution from a region adjacent to an $\infty$-segment and give it
to the innermost region.  This can be done by (a) following $\infty$-segments to
reach the unbounded region and then (b) by passing through the caps, illustrated as
in Figure~\ref{fig.cup}.
\begin{figure}[htpb!]
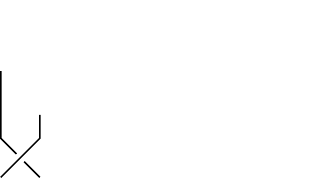
\caption{From a region adjacent to $s_\infty$ to the innermost region.}
\label{fig.cup}
\end{figure}
Similarly, we take one contribution from a region adjacent to a $0$-segment and give
it to the innermost region, by following $0$-segments and passing through the cups.
This results in a triple $\boldsymbol h$ such that
$\langle \boldsymbol f+\boldsymbol h \, , e \rangle = 2$
for all C and R-edges $e$.

For O and U-edges, we can use Lemma~\ref{lem.O} as we did in Section~\ref{sec.flat}.
This results in  another triple $\boldsymbol g $ such that 
\begin{equation*}
\langle \boldsymbol f+ \boldsymbol h + \boldsymbol g , e \rangle  
=\begin{cases}
  \langle \boldsymbol f+ \boldsymbol h , e \rangle & \textrm{if } e
  \textrm{ is a C or R-edge} \\
2 & \textrm{if } e \textrm{ is an O or U-edge} 
\end{cases}\,.
\end{equation*}
Namely, $\boldsymbol f+ \boldsymbol h+\boldsymbol g $ is a flattening of $\calT_{\Ddot}$.

%%%%%%%%%%%%%%%%%%%%%%%%%%%%%%%%%%%%%%%%%%%%%%%%%%%%%%%%%%%%%%%%%%%%%%%%%%%% 
%%%%%%%%%%%%%%%%%%%%%%%%%%%%%%%%%%%%%%%%%%%%%%%%%%%%%%%%%%%%%%%%%%%%%%%%%%%%

%\bibliographystyle{plain}
\bibliographystyle{hamsalpha}
\bibliography{biblio}
\end{document}

%% file: figures/tetrahedron.pdf_tex
%% Creator: Inkscape 1.0.2 (e86c8708, 2021-01-15), www.inkscape.org
%% Accompanies image file 'tetrahedron.pdf' (pdf, eps, ps)
\begingroup%
  \makeatletter%
  \providecommand\color[2][]{%
    \errmessage{(Inkscape) Color is used for the text in Inkscape, but the package 'color.sty' is not loaded}%
    \renewcommand\color[2][]{}%
  }%
  \providecommand\transparent[1]{%
    \errmessage{(Inkscape) Transparency is used (non-zero) for the text in Inkscape, but the package 'transparent.sty' is not loaded}%
    \renewcommand\transparent[1]{}%
  }%
  \providecommand\rotatebox[2]{#2}%
  \newcommand*\fsize{\dimexpr\f@size pt\relax}%
  \newcommand*\lineheight[1]{\fontsize{\fsize}{#1\fsize}\selectfont}%
  \ifx\svgwidth\undefined%
    \setlength{\unitlength}{102.80799325bp}%
    \ifx\svgscale\undefined%
      \relax%
    \else%
      \setlength{\unitlength}{\unitlength * \real{\svgscale}}%
    \fi%
  \else%
    \setlength{\unitlength}{\svgwidth}%
  \fi%
  \global\let\svgwidth\undefined%
  \global\let\svgscale\undefined%
  \makeatother%
  \begin{picture}(1,0.91189401)%
    \lineheight{1}%
    \setlength\tabcolsep{0pt}%
    \put(0,0){\includegraphics[width=\unitlength,page=1]{tetrahedron.pdf}}%
    \put(0.61843194,0.50078919){\makebox(0,0)[lt]{\lineheight{1.25}\smash{\begin{tabular}[t]{l}$z'$\end{tabular}}}}%
    \put(0,0){\includegraphics[width=\unitlength,page=2]{tetrahedron.pdf}}%
    \put(0.43716817,0.28094319){\makebox(0,0)[lt]{\lineheight{1.25}\smash{\begin{tabular}[t]{l}$z'$\end{tabular}}}}%
    \put(0,0){\includegraphics[width=\unitlength,page=3]{tetrahedron.pdf}}%
    \put(0.40954356,0.06260441){\color[rgb]{0,0,0}\makebox(0,0)[lt]{\lineheight{1.25}\smash{\begin{tabular}[t]{l}$z''$\end{tabular}}}}%
    \put(0,0){\includegraphics[width=\unitlength,page=4]{tetrahedron.pdf}}%
    \put(0.87169081,0.20134339){\makebox(0,0)[lt]{\lineheight{1.25}\smash{\begin{tabular}[t]{l}$z$\end{tabular}}}}%
    \put(0,0){\includegraphics[width=\unitlength,page=5]{tetrahedron.pdf}}%
    \put(0.24868823,0.53931904){\makebox(0,0)[lt]{\lineheight{1.25}\smash{\begin{tabular}[t]{l}$z$\end{tabular}}}}%
    \put(0,0){\includegraphics[width=\unitlength,page=6]{tetrahedron.pdf}}%
    \put(0.70124393,0.65818464){\makebox(0,0)[lt]{\lineheight{1.25}\smash{\begin{tabular}[t]{l}$z''$\end{tabular}}}}%
  \end{picture}%
\endgroup%

%% file: figures/example.pdf_tex
%% Creator: Inkscape 1.3.2 (091e20e, 2023-11-25, custom), www.inkscape.org
%% Accompanies image file 'example.pdf' (pdf, eps, ps)
\begingroup%
  \makeatletter%
  \providecommand\color[2][]{%
    \errmessage{(Inkscape) Color is used for the text in Inkscape, but the package 'color.sty' is not loaded}%
    \renewcommand\color[2][]{}%
  }%
  \providecommand\transparent[1]{%
    \errmessage{(Inkscape) Transparency is used (non-zero) for the text in Inkscape, but the package 'transparent.sty' is not loaded}%
    \renewcommand\transparent[1]{}%
  }%
  \providecommand\rotatebox[2]{#2}%
  \newcommand*\fsize{\dimexpr\f@size pt\relax}%
  \newcommand*\lineheight[1]{\fontsize{\fsize}{#1\fsize}\selectfont}%
  \ifx\svgwidth\undefined%
    \setlength{\unitlength}{86.78974842bp}%
    \ifx\svgscale\undefined%
      \relax%
    \else%
      \setlength{\unitlength}{\unitlength * \real{\svgscale}}%
    \fi%
  \else%
    \setlength{\unitlength}{\svgwidth}%
  \fi%
  \global\let\svgwidth\undefined%
  \global\let\svgscale\undefined%
  \makeatother%
  \begin{picture}(1,1.97193937)%
    \lineheight{1}%
    \setlength\tabcolsep{0pt}%
    \put(0,0){\includegraphics[width=\unitlength,page=1]{example.pdf}}%
    \put(0.54364316,1.32166727){\makebox(0,0)[lt]{\lineheight{1.25}\smash{\begin{tabular}[t]{l}$x_1$\end{tabular}}}}%
    \put(0.55801884,1.60136353){\makebox(0,0)[lt]{\lineheight{1.25}\smash{\begin{tabular}[t]{l}$\infty$\end{tabular}}}}%
    \put(0.37137331,0.96161112){\makebox(0,0)[lt]{\lineheight{1.25}\smash{\begin{tabular}[t]{l}$x_2$\end{tabular}}}}%
    \put(0.37678421,0.63156496){\makebox(0,0)[lt]{\lineheight{1.25}\smash{\begin{tabular}[t]{l}$x_3$\end{tabular}}}}%
    \put(0.41669641,0.28696296){\makebox(0,0)[lt]{\lineheight{1.25}\smash{\begin{tabular}[t]{l}$0$\end{tabular}}}}%
    \put(0.87532284,0.30530441){\makebox(0,0)[lt]{\lineheight{1.25}\smash{\begin{tabular}[t]{l}$1$\end{tabular}}}}%
    \put(0.88797629,0.62520852){\makebox(0,0)[lt]{\lineheight{1.25}\smash{\begin{tabular}[t]{l}$1$\end{tabular}}}}%
    \put(0.88296856,0.95343589){\makebox(0,0)[lt]{\lineheight{1.25}\smash{\begin{tabular}[t]{l}$1$\end{tabular}}}}%
    \put(0.89611302,1.61641201){\makebox(0,0)[lt]{\lineheight{1.25}\smash{\begin{tabular}[t]{l}$1$\end{tabular}}}}%
    \put(0.04543632,0.6273393){\makebox(0,0)[lt]{\lineheight{1.25}\smash{\begin{tabular}[t]{l}$1$\end{tabular}}}}%
    \put(0.08373667,1.58884377){\makebox(0,0)[lt]{\lineheight{1.25}\smash{\begin{tabular}[t]{l}$1$\end{tabular}}}}%
    \put(0,0){\includegraphics[width=\unitlength,page=2]{example.pdf}}%
  \end{picture}%
\endgroup%

%% file: figures/crossing.pdf_tex
%% Creator: Inkscape 1.0.2 (e86c8708, 2021-01-15), www.inkscape.org
%% Accompanies image file 'crossing.pdf' (pdf, eps, ps)
\begingroup%
  \makeatletter%
  \providecommand\color[2][]{%
    \errmessage{(Inkscape) Color is used for the text in Inkscape, but the package 'color.sty' is not loaded}%
    \renewcommand\color[2][]{}%
  }%
  \providecommand\transparent[1]{%
    \errmessage{(Inkscape) Transparency is used (non-zero) for the text in Inkscape, but the package 'transparent.sty' is not loaded}%
    \renewcommand\transparent[1]{}%
  }%
  \providecommand\rotatebox[2]{#2}%
  \newcommand*\fsize{\dimexpr\f@size pt\relax}%
  \newcommand*\lineheight[1]{\fontsize{\fsize}{#1\fsize}\selectfont}%
  \ifx\svgwidth\undefined%
    \setlength{\unitlength}{382.67716535bp}%
    \ifx\svgscale\undefined%
      \relax%
    \else%
      \setlength{\unitlength}{\unitlength * \real{\svgscale}}%
    \fi%
  \else%
    \setlength{\unitlength}{\svgwidth}%
  \fi%
  \global\let\svgwidth\undefined%
  \global\let\svgscale\undefined%
  \makeatother%
  \begin{picture}(1,0.15555556)%
    \lineheight{1}%
    \setlength\tabcolsep{0pt}%
    \put(0,0){\includegraphics[width=\unitlength,page=1]{crossing.pdf}}%
    \put(0.11765485,0.14343033){\makebox(0,0)[lt]{\lineheight{1.25}\smash{\begin{tabular}[t]{l}$x$\end{tabular}}}}%
    \put(0.26719821,0.14593571){\makebox(0,0)[lt]{\lineheight{1.25}\smash{\begin{tabular}[t]{l}$y$\end{tabular}}}}%
    \put(0.33458057,0.07102404){\makebox(0,0)[lt]{\lineheight{1.25}\smash{\begin{tabular}[t]{l}$\li\left(\dfrac{x}{y}\right)$\end{tabular}}}}%
    \put(0,0){\includegraphics[width=\unitlength,page=2]{crossing.pdf}}%
    \put(0.56271851,0.14297279){\makebox(0,0)[lt]{\lineheight{1.25}\smash{\begin{tabular}[t]{l}$x$\end{tabular}}}}%
    \put(0.71550776,0.1426332){\makebox(0,0)[lt]{\lineheight{1.25}\smash{\begin{tabular}[t]{l}$y$\end{tabular}}}}%
    \put(0.78470117,0.07326431){\makebox(0,0)[lt]{\lineheight{1.25}\smash{\begin{tabular}[t]{l}$-\li\left(\dfrac{y}{x}\right)$\end{tabular}}}}%
    \put(0,0){\includegraphics[width=\unitlength,page=3]{crossing.pdf}}%
  \end{picture}%
\endgroup%

%% file: figures/crossing2.pdf_tex
%% Creator: Inkscape 1.0.2 (e86c8708, 2021-01-15), www.inkscape.org
%% Accompanies image file 'crossing2.pdf' (pdf, eps, ps)
\begingroup%
  \makeatletter%
  \providecommand\color[2][]{%
    \errmessage{(Inkscape) Color is used for the text in Inkscape, but the package 'color.sty' is not loaded}%
    \renewcommand\color[2][]{}%
  }%
  \providecommand\transparent[1]{%
    \errmessage{(Inkscape) Transparency is used (non-zero) for the text in Inkscape, but the package 'transparent.sty' is not loaded}%
    \renewcommand\transparent[1]{}%
  }%
  \providecommand\rotatebox[2]{#2}%
  \newcommand*\fsize{\dimexpr\f@size pt\relax}%
  \newcommand*\lineheight[1]{\fontsize{\fsize}{#1\fsize}\selectfont}%
  \ifx\svgwidth\undefined%
    \setlength{\unitlength}{268.04196816bp}%
    \ifx\svgscale\undefined%
      \relax%
    \else%
      \setlength{\unitlength}{\unitlength * \real{\svgscale}}%
    \fi%
  \else%
    \setlength{\unitlength}{\svgwidth}%
  \fi%
  \global\let\svgwidth\undefined%
  \global\let\svgscale\undefined%
  \makeatother%
  \begin{picture}(1,0.29418902)%
    \lineheight{1}%
    \setlength\tabcolsep{0pt}%
    \put(0,0){\includegraphics[width=\unitlength,page=1]{crossing2.pdf}}%
    \put(0.59137521,0.14824103){\makebox(0,0)[lt]{\lineheight{1.25}\smash{\begin{tabular}[t]{l}$x$\end{tabular}}}}%
    \put(0.84511155,0.15160522){\makebox(0,0)[lt]{\lineheight{1.25}\smash{\begin{tabular}[t]{l}$y$\end{tabular}}}}%
    \put(0.72061823,-0.01790883){\makebox(0,0)[lt]{\lineheight{1.25}\smash{\begin{tabular}[t]{l}$0$\end{tabular}}}}%
    \put(0.71098434,0.28733053){\makebox(0,0)[lt]{\lineheight{1.25}\smash{\begin{tabular}[t]{l}$\infty$\end{tabular}}}}%
    \put(0,0){\includegraphics[width=\unitlength,page=2]{crossing2.pdf}}%
    \put(0.19416888,0.24208018){\makebox(0,0)[lt]{\lineheight{1.25}\smash{\begin{tabular}[t]{l}$x$\end{tabular}}}}%
    \put(0.41840878,0.2474471){\makebox(0,0)[lt]{\lineheight{1.25}\smash{\begin{tabular}[t]{l}$y$\end{tabular}}}}%
    \put(0,0){\includegraphics[width=\unitlength,page=3]{crossing2.pdf}}%
  \end{picture}%
\endgroup%

%% file: figures/omega1.pdf_tex
%% Creator: Inkscape 1.3.2 (091e20e, 2023-11-25, custom), www.inkscape.org
%% Accompanies image file 'omega1.pdf' (pdf, eps, ps)
\begingroup%
  \makeatletter%
  \providecommand\color[2][]{%
    \errmessage{(Inkscape) Color is used for the text in Inkscape, but the package 'color.sty' is not loaded}%
    \renewcommand\color[2][]{}%
  }%
  \providecommand\transparent[1]{%
    \errmessage{(Inkscape) Transparency is used (non-zero) for the text in Inkscape, but the package 'transparent.sty' is not loaded}%
    \renewcommand\transparent[1]{}%
  }%
  \providecommand\rotatebox[2]{#2}%
  \newcommand*\fsize{\dimexpr\f@size pt\relax}%
  \newcommand*\lineheight[1]{\fontsize{\fsize}{#1\fsize}\selectfont}%
  \ifx\svgwidth\undefined%
    \setlength{\unitlength}{428.9865713bp}%
    \ifx\svgscale\undefined%
      \relax%
    \else%
      \setlength{\unitlength}{\unitlength * \real{\svgscale}}%
    \fi%
  \else%
    \setlength{\unitlength}{\svgwidth}%
  \fi%
  \global\let\svgwidth\undefined%
  \global\let\svgscale\undefined%
  \makeatother%
  \begin{picture}(1,0.55505289)%
    \lineheight{1}%
    \setlength\tabcolsep{0pt}%
    \put(0,0){\includegraphics[width=\unitlength,page=1]{omega1.pdf}}%
    \put(0.45501653,0.35279767){\makebox(0,0)[lt]{\lineheight{1.25}\smash{\begin{tabular}[t]{l}$\red{\left(1-\dfrac{x'}{x}\right)}, \quad \blue{\left(1-\dfrac{y}{y'}\right)}$\end{tabular}}}}%
    \put(0,0){\includegraphics[width=\unitlength,page=2]{omega1.pdf}}%
    \put(0.17356456,0.2049344){\makebox(0,0)[lt]{\lineheight{1.25}\smash{\begin{tabular}[t]{l}$x$\end{tabular}}}}%
    \put(0.10210146,0.21074402){\makebox(0,0)[lt]{\lineheight{1.25}\smash{\begin{tabular}[t]{l}$z$\end{tabular}}}}%
    \put(0.17598792,0.14189831){\makebox(0,0)[lt]{\lineheight{1.25}\smash{\begin{tabular}[t]{l}$z'$\end{tabular}}}}%
    \put(0,0){\includegraphics[width=\unitlength,page=3]{omega1.pdf}}%
    \put(0.60900276,0.22683656){\makebox(0,0)[lt]{\lineheight{1.25}\smash{\begin{tabular}[t]{l}$\red{\left(1-\dfrac{z}{x}\right)}, \quad \blue{\left(1-\dfrac{z'}{x}\right)^{-1}}$\end{tabular}}}}%
    \put(0,0){\includegraphics[width=\unitlength,page=4]{omega1.pdf}}%
    \put(0.24154483,0.39917988){\makebox(0,0)[lt]{\lineheight{1.25}\smash{\begin{tabular}[t]{l}$x$\end{tabular}}}}%
    \put(0.36282043,0.40026462){\makebox(0,0)[lt]{\lineheight{1.25}\smash{\begin{tabular}[t]{l}$y$\end{tabular}}}}%
    \put(0.23736859,0.3013203){\makebox(0,0)[lt]{\lineheight{1.25}\smash{\begin{tabular}[t]{l}$x'$\end{tabular}}}}%
    \put(0.36566274,0.29906107){\makebox(0,0)[lt]{\lineheight{1.25}\smash{\begin{tabular}[t]{l}$y'$\end{tabular}}}}%
    \put(0,0){\includegraphics[width=\unitlength,page=5]{omega1.pdf}}%
    \put(0.24062353,0.5451454){\makebox(0,0)[lt]{\lineheight{1.25}\smash{\begin{tabular}[t]{l}$x$\end{tabular}}}}%
    \put(0.36342363,0.54578272){\makebox(0,0)[lt]{\lineheight{1.25}\smash{\begin{tabular}[t]{l}$y$\end{tabular}}}}%
    \put(0.23528445,0.44728576){\makebox(0,0)[lt]{\lineheight{1.25}\smash{\begin{tabular}[t]{l}$x'$\end{tabular}}}}%
    \put(0.36357855,0.44851948){\makebox(0,0)[lt]{\lineheight{1.25}\smash{\begin{tabular}[t]{l}$y'$\end{tabular}}}}%
    \put(0,0){\includegraphics[width=\unitlength,page=6]{omega1.pdf}}%
    \put(0.30135275,0.19883963){\makebox(0,0)[lt]{\lineheight{1.25}\smash{\begin{tabular}[t]{l}and\end{tabular}}}}%
    \put(0,0){\includegraphics[width=\unitlength,page=7]{omega1.pdf}}%
    \put(0.44679609,0.20529152){\makebox(0,0)[lt]{\lineheight{1.25}\smash{\begin{tabular}[t]{l}$x$\end{tabular}}}}%
    \put(0.37430983,0.1997141){\makebox(0,0)[lt]{\lineheight{1.25}\smash{\begin{tabular}[t]{l}$z$\end{tabular}}}}%
    \put(0.45200772,0.26773566){\makebox(0,0)[lt]{\lineheight{1.25}\smash{\begin{tabular}[t]{l}$z'$\end{tabular}}}}%
    \put(0,0){\includegraphics[width=\unitlength,page=8]{omega1.pdf}}%
    \put(0.17140268,0.04883293){\makebox(0,0)[lt]{\lineheight{1.25}\smash{\begin{tabular}[t]{l}$x$\end{tabular}}}}%
    \put(0.10208636,0.05751329){\makebox(0,0)[lt]{\lineheight{1.25}\smash{\begin{tabular}[t]{l}$z$\end{tabular}}}}%
    \put(0.1765644,-0.00751946){\makebox(0,0)[lt]{\lineheight{1.25}\smash{\begin{tabular}[t]{l}$z'$\end{tabular}}}}%
    \put(0.60556861,0.07610572){\makebox(0,0)[lt]{\lineheight{1.25}\smash{\begin{tabular}[t]{l}$\red{\left(1-\dfrac{x}{z}\right)},\quad \blue{\left(1-\dfrac{x}{z'}\right)^{-1}}$\end{tabular}}}}%
    \put(0.44693165,0.05505383){\makebox(0,0)[lt]{\lineheight{1.25}\smash{\begin{tabular}[t]{l}$x$\end{tabular}}}}%
    \put(0.37426647,0.04808913){\makebox(0,0)[lt]{\lineheight{1.25}\smash{\begin{tabular}[t]{l}$z$\end{tabular}}}}%
    \put(0.45198528,0.11302802){\makebox(0,0)[lt]{\lineheight{1.25}\smash{\begin{tabular}[t]{l}$z'$\end{tabular}}}}%
    \put(0,0){\includegraphics[width=\unitlength,page=9]{omega1.pdf}}%
    \put(0.44863698,0.49642059){\makebox(0,0)[lt]{\lineheight{1.25}\smash{\begin{tabular}[t]{l}$\red{\left(1-\dfrac{x}{x'}\right)}, \quad \blue{\left(1-\dfrac{y'}{y}\right)}$\end{tabular}}}}%
    \put(0,0){\includegraphics[width=\unitlength,page=10]{omega1.pdf}}%
    \put(0.30182538,0.04508964){\makebox(0,0)[lt]{\lineheight{1.25}\smash{\begin{tabular}[t]{l}and\end{tabular}}}}%
    \put(0,0){\includegraphics[width=\unitlength,page=11]{omega1.pdf}}%
  \end{picture}%
\endgroup%

%% file: figures/omega2.pdf_tex
%% Creator: Inkscape 1.3.2 (091e20e, 2023-11-25, custom), www.inkscape.org
%% Accompanies image file 'omega2.pdf' (pdf, eps, ps)
\begingroup%
  \makeatletter%
  \providecommand\color[2][]{%
    \errmessage{(Inkscape) Color is used for the text in Inkscape, but the package 'color.sty' is not loaded}%
    \renewcommand\color[2][]{}%
  }%
  \providecommand\transparent[1]{%
    \errmessage{(Inkscape) Transparency is used (non-zero) for the text in Inkscape, but the package 'transparent.sty' is not loaded}%
    \renewcommand\transparent[1]{}%
  }%
  \providecommand\rotatebox[2]{#2}%
  \newcommand*\fsize{\dimexpr\f@size pt\relax}%
  \newcommand*\lineheight[1]{\fontsize{\fsize}{#1\fsize}\selectfont}%
  \ifx\svgwidth\undefined%
    \setlength{\unitlength}{311.81102362bp}%
    \ifx\svgscale\undefined%
      \relax%
    \else%
      \setlength{\unitlength}{\unitlength * \real{\svgscale}}%
    \fi%
  \else%
    \setlength{\unitlength}{\svgwidth}%
  \fi%
  \global\let\svgwidth\undefined%
  \global\let\svgscale\undefined%
  \makeatother%
  \begin{picture}(1,0.16066373)%
    \lineheight{1}%
    \setlength\tabcolsep{0pt}%
    \put(0,0){\includegraphics[width=\unitlength,page=1]{omega2.pdf}}%
    \put(0.56817649,0.14218339){\makebox(0,0)[lt]{\lineheight{1.25}\smash{\begin{tabular}[t]{l}$x$\end{tabular}}}}%
    \put(0.73502635,0.14367591){\makebox(0,0)[lt]{\lineheight{1.25}\smash{\begin{tabular}[t]{l}$y$\end{tabular}}}}%
    \put(0.56243086,0.0075492){\makebox(0,0)[lt]{\lineheight{1.25}\smash{\begin{tabular}[t]{l}$x'$\end{tabular}}}}%
    \put(0.73893676,0.00444099){\makebox(0,0)[lt]{\lineheight{1.25}\smash{\begin{tabular}[t]{l}$y'$\end{tabular}}}}%
    \put(0,0){\includegraphics[width=\unitlength,page=2]{omega2.pdf}}%
    \put(0.00567454,0.14232655){\makebox(0,0)[lt]{\lineheight{1.25}\smash{\begin{tabular}[t]{l}$x$\end{tabular}}}}%
    \put(0.17462173,0.1432033){\makebox(0,0)[lt]{\lineheight{1.25}\smash{\begin{tabular}[t]{l}$y$\end{tabular}}}}%
    \put(-0.00167076,0.00769208){\makebox(0,0)[lt]{\lineheight{1.25}\smash{\begin{tabular}[t]{l}$x'$\end{tabular}}}}%
    \put(0.17483514,0.00938948){\makebox(0,0)[lt]{\lineheight{1.25}\smash{\begin{tabular}[t]{l}$y'$\end{tabular}}}}%
    \put(0.29185716,0.07529115){\makebox(0,0)[lt]{\lineheight{1.25}\smash{\begin{tabular}[t]{l}$\left(\dfrac{x'}{x}\right)^2$\end{tabular}}}}%
    \put(0.84651677,0.07954072){\makebox(0,0)[lt]{\lineheight{1.25}\smash{\begin{tabular}[t]{l}$\left(\dfrac{y'}{y}\right)^2$\end{tabular}}}}%
    \put(0,0){\includegraphics[width=\unitlength,page=3]{omega2.pdf}}%
  \end{picture}%
\endgroup%

%% file: figures/octahedronatc.pdf_tex
%% Creator: Inkscape 1.0.2 (e86c8708, 2021-01-15), www.inkscape.org
%% Accompanies image file 'octahedronatc.pdf' (pdf, eps, ps)
\begingroup%
  \makeatletter%
  \providecommand\color[2][]{%
    \errmessage{(Inkscape) Color is used for the text in Inkscape, but the package 'color.sty' is not loaded}%
    \renewcommand\color[2][]{}%
  }%
  \providecommand\transparent[1]{%
    \errmessage{(Inkscape) Transparency is used (non-zero) for the text in Inkscape, but the package 'transparent.sty' is not loaded}%
    \renewcommand\transparent[1]{}%
  }%
  \providecommand\rotatebox[2]{#2}%
  \newcommand*\fsize{\dimexpr\f@size pt\relax}%
  \newcommand*\lineheight[1]{\fontsize{\fsize}{#1\fsize}\selectfont}%
  \ifx\svgwidth\undefined%
    \setlength{\unitlength}{190.47641544bp}%
    \ifx\svgscale\undefined%
      \relax%
    \else%
      \setlength{\unitlength}{\unitlength * \real{\svgscale}}%
    \fi%
  \else%
    \setlength{\unitlength}{\svgwidth}%
  \fi%
  \global\let\svgwidth\undefined%
  \global\let\svgscale\undefined%
  \makeatother%
  \begin{picture}(1,0.51379036)%
    \lineheight{1}%
    \setlength\tabcolsep{0pt}%
    \put(0,0){\includegraphics[width=\unitlength,page=1]{octahedronatc.pdf}}%
  \end{picture}%
\endgroup%

%% file: figures/octahedron.pdf_tex
%% Creator: Inkscape 1.0.2 (e86c8708, 2021-01-15), www.inkscape.org
%% Accompanies image file 'octahedron.pdf' (pdf, eps, ps)
\begingroup%
  \makeatletter%
  \providecommand\color[2][]{%
    \errmessage{(Inkscape) Color is used for the text in Inkscape, but the package 'color.sty' is not loaded}%
    \renewcommand\color[2][]{}%
  }%
  \providecommand\transparent[1]{%
    \errmessage{(Inkscape) Transparency is used (non-zero) for the text in Inkscape, but the package 'transparent.sty' is not loaded}%
    \renewcommand\transparent[1]{}%
  }%
  \providecommand\rotatebox[2]{#2}%
  \newcommand*\fsize{\dimexpr\f@size pt\relax}%
  \newcommand*\lineheight[1]{\fontsize{\fsize}{#1\fsize}\selectfont}%
  \ifx\svgwidth\undefined%
    \setlength{\unitlength}{345.67309282bp}%
    \ifx\svgscale\undefined%
      \relax%
    \else%
      \setlength{\unitlength}{\unitlength * \real{\svgscale}}%
    \fi%
  \else%
    \setlength{\unitlength}{\svgwidth}%
  \fi%
  \global\let\svgwidth\undefined%
  \global\let\svgscale\undefined%
  \makeatother%
  \begin{picture}(1,0.33888542)%
    \lineheight{1}%
    \setlength\tabcolsep{0pt}%
    \put(0,0){\includegraphics[width=\unitlength,page=1]{octahedron.pdf}}%
    \put(0.77474084,0.15645441){\makebox(0,0)[lt]{\lineheight{1.25}\smash{\begin{tabular}[t]{l}$+$\end{tabular}}}}%
    \put(0.82696148,0.1089888){\makebox(0,0)[lt]{\lineheight{1.25}\smash{\begin{tabular}[t]{l}$-$\end{tabular}}}}%
    \put(0.84357005,0.23078555){\makebox(0,0)[lt]{\lineheight{1.25}\smash{\begin{tabular}[t]{l}$+$\end{tabular}}}}%
    \put(0.9044693,0.16948802){\makebox(0,0)[lt]{\lineheight{1.25}\smash{\begin{tabular}[t]{l}$-$\end{tabular}}}}%
    \put(0,0){\includegraphics[width=\unitlength,page=2]{octahedron.pdf}}%
  \end{picture}%
\endgroup%

%% file: figures/region2.pdf_tex
%% Creator: Inkscape 1.0.2 (e86c8708, 2021-01-15), www.inkscape.org
%% Accompanies image file 'region2.pdf' (pdf, eps, ps)
\begingroup%
  \makeatletter%
  \providecommand\color[2][]{%
    \errmessage{(Inkscape) Color is used for the text in Inkscape, but the package 'color.sty' is not loaded}%
    \renewcommand\color[2][]{}%
  }%
  \providecommand\transparent[1]{%
    \errmessage{(Inkscape) Transparency is used (non-zero) for the text in Inkscape, but the package 'transparent.sty' is not loaded}%
    \renewcommand\transparent[1]{}%
  }%
  \providecommand\rotatebox[2]{#2}%
  \newcommand*\fsize{\dimexpr\f@size pt\relax}%
  \newcommand*\lineheight[1]{\fontsize{\fsize}{#1\fsize}\selectfont}%
  \ifx\svgwidth\undefined%
    \setlength{\unitlength}{356.92804532bp}%
    \ifx\svgscale\undefined%
      \relax%
    \else%
      \setlength{\unitlength}{\unitlength * \real{\svgscale}}%
    \fi%
  \else%
    \setlength{\unitlength}{\svgwidth}%
  \fi%
  \global\let\svgwidth\undefined%
  \global\let\svgscale\undefined%
  \makeatother%
  \begin{picture}(1,0.43089208)%
    \lineheight{1}%
    \setlength\tabcolsep{0pt}%
    \put(0,0){\includegraphics[width=\unitlength,page=1]{region2.pdf}}%
    \put(0.61794712,0.31893231){\makebox(0,0)[lt]{\lineheight{1.25}\smash{\begin{tabular}[t]{l}$\pm$\end{tabular}}}}%
    \put(0.61799789,0.0887149){\makebox(0,0)[lt]{\lineheight{1.25}\smash{\begin{tabular}[t]{l}$\pm$\end{tabular}}}}%
    \put(0,0){\includegraphics[width=\unitlength,page=2]{region2.pdf}}%
    \put(0.78841713,0.32275433){\makebox(0,0)[lt]{\lineheight{1.25}\smash{\begin{tabular}[t]{l}$\pm$\end{tabular}}}}%
    \put(0,0){\includegraphics[width=\unitlength,page=3]{region2.pdf}}%
    \put(0.31421397,0.15443603){\makebox(0,0)[lt]{\lineheight{1.25}\smash{\begin{tabular}[t]{l}$u$\end{tabular}}}}%
    \put(0.67844361,0.25813748){\makebox(0,0)[lt]{\lineheight{1.25}\smash{\begin{tabular}[t]{l}$o$\end{tabular}}}}%
    \put(0,0){\includegraphics[width=\unitlength,page=4]{region2.pdf}}%
    \put(0.49668173,0.20757561){\makebox(0,0)[lt]{\lineheight{1.25}\smash{\begin{tabular}[t]{l}$s$\end{tabular}}}}%
    \put(0,0){\includegraphics[width=\unitlength,page=5]{region2.pdf}}%
    \put(0.65653499,0.20540627){\makebox(0,0)[lt]{\lineheight{1.25}\smash{\begin{tabular}[t]{l}$\cdots$\end{tabular}}}}%
    \put(0.49511339,0.29737143){\makebox(0,0)[lt]{\lineheight{1.25}\smash{\begin{tabular}[t]{l}$r_1$\end{tabular}}}}%
    \put(0.49838615,0.12527503){\makebox(0,0)[lt]{\lineheight{1.25}\smash{\begin{tabular}[t]{l}$r_2$\end{tabular}}}}%
    \put(0,0){\includegraphics[width=\unitlength,page=6]{region2.pdf}}%
    \put(0.79119475,0.09106123){\makebox(0,0)[lt]{\lineheight{1.25}\smash{\begin{tabular}[t]{l}$\pm$\end{tabular}}}}%
    \put(0.35813057,0.32157832){\makebox(0,0)[lt]{\lineheight{1.25}\smash{\begin{tabular}[t]{l}$\pm$\end{tabular}}}}%
    \put(0.35681793,0.09136091){\makebox(0,0)[lt]{\lineheight{1.25}\smash{\begin{tabular}[t]{l}$\pm$\end{tabular}}}}%
    \put(0,0){\includegraphics[width=\unitlength,page=7]{region2.pdf}}%
    \put(0.26186249,0.32155833){\makebox(0,0)[lt]{\lineheight{1.25}\smash{\begin{tabular}[t]{l}$\pm$\end{tabular}}}}%
    \put(0,0){\includegraphics[width=\unitlength,page=8]{region2.pdf}}%
    \put(0.26464016,0.09134091){\makebox(0,0)[lt]{\lineheight{1.25}\smash{\begin{tabular}[t]{l}$\pm$\end{tabular}}}}%
    \put(0,0){\includegraphics[width=\unitlength,page=9]{region2.pdf}}%
    \put(0.18205759,0.32112983){\makebox(0,0)[lt]{\lineheight{1.25}\smash{\begin{tabular}[t]{l}$\pm$\end{tabular}}}}%
    \put(0.1848352,0.09091235){\makebox(0,0)[lt]{\lineheight{1.25}\smash{\begin{tabular}[t]{l}$\pm$\end{tabular}}}}%
    \put(0,0){\includegraphics[width=\unitlength,page=10]{region2.pdf}}%
    \put(0.302606,0.2050915){\makebox(0,0)[lt]{\lineheight{1.25}\smash{\begin{tabular}[t]{l}$\cdots$\end{tabular}}}}%
    \put(0.71043142,0.32134141){\makebox(0,0)[lt]{\lineheight{1.25}\smash{\begin{tabular}[t]{l}$\pm$\end{tabular}}}}%
    \put(0.71048219,0.09112388){\makebox(0,0)[lt]{\lineheight{1.25}\smash{\begin{tabular}[t]{l}$\pm$\end{tabular}}}}%
  \end{picture}%
\endgroup%

%% file: figures/examplecc.pdf_tex
%% Creator: Inkscape 1.3.2 (091e20e, 2023-11-25, custom), www.inkscape.org
%% Accompanies image file 'examplecc.pdf' (pdf, eps, ps)
\begingroup%
  \makeatletter%
  \providecommand\color[2][]{%
    \errmessage{(Inkscape) Color is used for the text in Inkscape, but the package 'color.sty' is not loaded}%
    \renewcommand\color[2][]{}%
  }%
  \providecommand\transparent[1]{%
    \errmessage{(Inkscape) Transparency is used (non-zero) for the text in Inkscape, but the package 'transparent.sty' is not loaded}%
    \renewcommand\transparent[1]{}%
  }%
  \providecommand\rotatebox[2]{#2}%
  \newcommand*\fsize{\dimexpr\f@size pt\relax}%
  \newcommand*\lineheight[1]{\fontsize{\fsize}{#1\fsize}\selectfont}%
  \ifx\svgwidth\undefined%
    \setlength{\unitlength}{345.00072089bp}%
    \ifx\svgscale\undefined%
      \relax%
    \else%
      \setlength{\unitlength}{\unitlength * \real{\svgscale}}%
    \fi%
  \else%
    \setlength{\unitlength}{\svgwidth}%
  \fi%
  \global\let\svgwidth\undefined%
  \global\let\svgscale\undefined%
  \makeatother%
  \begin{picture}(1,0.55161539)%
    \lineheight{1}%
    \setlength\tabcolsep{0pt}%
    \put(0.15415258,0.45897808){\makebox(0,0)[lt]{\lineheight{1.25}\smash{\begin{tabular}[t]{l}$\frac{o_1}{u_1}$\end{tabular}}}}%
    \put(0.02298916,0.19418405){\makebox(0,0)[lt]{\lineheight{1.25}\smash{\begin{tabular}[t]{l}$\frac{o_1}{u_2}$\end{tabular}}}}%
    \put(0,0){\includegraphics[width=\unitlength,page=1]{examplecc.pdf}}%
    \put(0.90682741,0.21755882){\makebox(0,0)[lt]{\lineheight{1.25}\smash{\begin{tabular}[t]{l}R-edges\end{tabular}}}}%
    \put(0.90489481,0.28214815){\makebox(0,0)[lt]{\lineheight{1.25}\smash{\begin{tabular}[t]{l}C-edges\end{tabular}}}}%
    \put(0,0){\includegraphics[width=\unitlength,page=2]{examplecc.pdf}}%
    \put(0.23618341,0.13084348){\makebox(0,0)[lt]{\lineheight{1.25}\smash{\begin{tabular}[t]{l}$\frac{o_2}{u_2}$\end{tabular}}}}%
    \put(0.11803957,0.21164391){\makebox(0,0)[lt]{\lineheight{1.25}\smash{\begin{tabular}[t]{l}$\frac{o_2}{u_3}$\end{tabular}}}}%
    \put(0.23568744,0.29509144){\makebox(0,0)[lt]{\lineheight{1.25}\smash{\begin{tabular}[t]{l}$\frac{o_3}{u_3}$\end{tabular}}}}%
    \put(0.15181089,0.39179869){\makebox(0,0)[lt]{\lineheight{1.25}\smash{\begin{tabular}[t]{l}$\frac{o_3}{u_4}$\end{tabular}}}}%
    \put(0.02817769,0.45720934){\makebox(0,0)[lt]{\lineheight{1.25}\smash{\begin{tabular}[t]{l}$\frac{o_4}{u_4}$\end{tabular}}}}%
    \put(0.24268226,0.48955945){\makebox(0,0)[lt]{\lineheight{1.25}\smash{\begin{tabular}[t]{l}$\frac{o_4}{u_5}$\end{tabular}}}}%
    \put(0.11663957,0.29444596){\makebox(0,0)[lt]{\lineheight{1.25}\smash{\begin{tabular}[t]{l}$\frac{o_5}{u_5}$\end{tabular}}}}%
    \put(0.23568744,0.21136427){\makebox(0,0)[lt]{\lineheight{1.25}\smash{\begin{tabular}[t]{l}$\frac{o_5}{u_6}$\end{tabular}}}}%
    \put(0.11668355,0.13155396){\makebox(0,0)[lt]{\lineheight{1.25}\smash{\begin{tabular}[t]{l}$\frac{o_6}{u_6}$\end{tabular}}}}%
    \put(-0.03993632,0.49078351){\makebox(0,0)[lt]{\lineheight{1.25}\smash{\begin{tabular}[t]{l}$\frac{o_6}{u_1}$\end{tabular}}}}%
    \put(0.64740298,0.43855149){\makebox(0,0)[lt]{\lineheight{1.25}\smash{\begin{tabular}[t]{l}$\frac{o_1}{}$\end{tabular}}}}%
    \put(0.51623944,0.17375733){\makebox(0,0)[lt]{\lineheight{1.25}\smash{\begin{tabular}[t]{l}$\frac{o_1}{u_2}$\end{tabular}}}}%
    \put(0,0){\includegraphics[width=\unitlength,page=3]{examplecc.pdf}}%
    \put(0.72943369,0.11041708){\makebox(0,0)[lt]{\lineheight{1.25}\smash{\begin{tabular}[t]{l}$\frac{o_2}{u_2}$\end{tabular}}}}%
    \put(0.61128978,0.19121726){\makebox(0,0)[lt]{\lineheight{1.25}\smash{\begin{tabular}[t]{l}$\frac{o_2}{u_3}$\end{tabular}}}}%
    \put(0.72893772,0.27466497){\makebox(0,0)[lt]{\lineheight{1.25}\smash{\begin{tabular}[t]{l}$\frac{o_3}{u_3}$\end{tabular}}}}%
    \put(0.64506111,0.3713721){\makebox(0,0)[lt]{\lineheight{1.25}\smash{\begin{tabular}[t]{l}$\frac{o_3}{u_4}$\end{tabular}}}}%
    \put(0.521428,0.43678274){\makebox(0,0)[lt]{\lineheight{1.25}\smash{\begin{tabular}[t]{l}$\frac{o_4}{u_4}$\end{tabular}}}}%
    \put(0.73593257,0.46913286){\makebox(0,0)[lt]{\lineheight{1.25}\smash{\begin{tabular}[t]{l}$\frac{o_4}{u_5}$\end{tabular}}}}%
    \put(0.60988982,0.2740195){\makebox(0,0)[lt]{\lineheight{1.25}\smash{\begin{tabular}[t]{l}$\frac{o_5}{u_5}$\end{tabular}}}}%
    \put(0.72893772,0.19093774){\makebox(0,0)[lt]{\lineheight{1.25}\smash{\begin{tabular}[t]{l}$\frac{o_5}{u_6}$\end{tabular}}}}%
    \put(0.60993377,0.11112756){\makebox(0,0)[lt]{\lineheight{1.25}\smash{\begin{tabular}[t]{l}$\frac{}{u_6}$\end{tabular}}}}%
    \put(0,0){\includegraphics[width=\unitlength,page=4]{examplecc.pdf}}%
  \end{picture}%
\endgroup%

%% file: figures/planardiagram.pdf_tex
%% Creator: Inkscape 1.0.2 (e86c8708, 2021-01-15), www.inkscape.org
%% Accompanies image file 'planardiagram.pdf' (pdf, eps, ps)
\begingroup%
  \makeatletter%
  \providecommand\color[2][]{%
    \errmessage{(Inkscape) Color is used for the text in Inkscape, but the package 'color.sty' is not loaded}%
    \renewcommand\color[2][]{}%
  }%
  \providecommand\transparent[1]{%
    \errmessage{(Inkscape) Transparency is used (non-zero) for the text in Inkscape, but the package 'transparent.sty' is not loaded}%
    \renewcommand\transparent[1]{}%
  }%
  \providecommand\rotatebox[2]{#2}%
  \newcommand*\fsize{\dimexpr\f@size pt\relax}%
  \newcommand*\lineheight[1]{\fontsize{\fsize}{#1\fsize}\selectfont}%
  \ifx\svgwidth\undefined%
    \setlength{\unitlength}{234.78144368bp}%
    \ifx\svgscale\undefined%
      \relax%
    \else%
      \setlength{\unitlength}{\unitlength * \real{\svgscale}}%
    \fi%
  \else%
    \setlength{\unitlength}{\svgwidth}%
  \fi%
  \global\let\svgwidth\undefined%
  \global\let\svgscale\undefined%
  \makeatother%
  \begin{picture}(1,0.90840288)%
    \lineheight{1}%
    \setlength\tabcolsep{0pt}%
    \put(0,0){\includegraphics[width=\unitlength,page=1]{planardiagram.pdf}}%
    \put(0.16974862,0.30401226){\makebox(0,0)[lt]{\lineheight{1.25}\smash{\begin{tabular}[t]{l}$\vdots$\end{tabular}}}}%
    \put(0,0){\includegraphics[width=\unitlength,page=2]{planardiagram.pdf}}%
    \put(0.45579776,0.42491449){\makebox(0,0)[lt]{\lineheight{1.25}\smash{\begin{tabular}[t]{l}or\end{tabular}}}}%
    \put(0.04901309,0.30401226){\makebox(0,0)[lt]{\lineheight{1.25}\smash{\begin{tabular}[t]{l}$\vdots$\end{tabular}}}}%
    \put(0.29048415,0.30401226){\makebox(0,0)[lt]{\lineheight{1.25}\smash{\begin{tabular}[t]{l}$\vdots$\end{tabular}}}}%
    \put(0.78056874,0.30334388){\makebox(0,0)[lt]{\lineheight{1.25}\smash{\begin{tabular}[t]{l}$\vdots$\end{tabular}}}}%
    \put(0.65983319,0.30334388){\makebox(0,0)[lt]{\lineheight{1.25}\smash{\begin{tabular}[t]{l}$\vdots$\end{tabular}}}}%
    \put(0.90130416,0.30334388){\makebox(0,0)[lt]{\lineheight{1.25}\smash{\begin{tabular}[t]{l}$\vdots$\end{tabular}}}}%
  \end{picture}%
\endgroup%

%% file: figures/octahedronatd.pdf_tex
%% Creator: Inkscape 1.0.2 (e86c8708, 2021-01-15), www.inkscape.org
%% Accompanies image file 'octahedronatd.pdf' (pdf, eps, ps)
\begingroup%
  \makeatletter%
  \providecommand\color[2][]{%
    \errmessage{(Inkscape) Color is used for the text in Inkscape, but the package 'color.sty' is not loaded}%
    \renewcommand\color[2][]{}%
  }%
  \providecommand\transparent[1]{%
    \errmessage{(Inkscape) Transparency is used (non-zero) for the text in Inkscape, but the package 'transparent.sty' is not loaded}%
    \renewcommand\transparent[1]{}%
  }%
  \providecommand\rotatebox[2]{#2}%
  \newcommand*\fsize{\dimexpr\f@size pt\relax}%
  \newcommand*\lineheight[1]{\fontsize{\fsize}{#1\fsize}\selectfont}%
  \ifx\svgwidth\undefined%
    \setlength{\unitlength}{353.11762772bp}%
    \ifx\svgscale\undefined%
      \relax%
    \else%
      \setlength{\unitlength}{\unitlength * \real{\svgscale}}%
    \fi%
  \else%
    \setlength{\unitlength}{\svgwidth}%
  \fi%
  \global\let\svgwidth\undefined%
  \global\let\svgscale\undefined%
  \makeatother%
  \begin{picture}(1,0.40866445)%
    \lineheight{1}%
    \setlength\tabcolsep{0pt}%
    \put(0,0){\includegraphics[width=\unitlength,page=1]{octahedronatd.pdf}}%
  \end{picture}%
\endgroup%

%% file: figures/infty.pdf_tex
%% Creator: Inkscape 1.3.2 (091e20e, 2023-11-25, custom), www.inkscape.org
%% Accompanies image file 'infty.pdf' (pdf, eps, ps)
\begingroup%
  \makeatletter%
  \providecommand\color[2][]{%
    \errmessage{(Inkscape) Color is used for the text in Inkscape, but the package 'color.sty' is not loaded}%
    \renewcommand\color[2][]{}%
  }%
  \providecommand\transparent[1]{%
    \errmessage{(Inkscape) Transparency is used (non-zero) for the text in Inkscape, but the package 'transparent.sty' is not loaded}%
    \renewcommand\transparent[1]{}%
  }%
  \providecommand\rotatebox[2]{#2}%
  \newcommand*\fsize{\dimexpr\f@size pt\relax}%
  \newcommand*\lineheight[1]{\fontsize{\fsize}{#1\fsize}\selectfont}%
  \ifx\svgwidth\undefined%
    \setlength{\unitlength}{255.65267728bp}%
    \ifx\svgscale\undefined%
      \relax%
    \else%
      \setlength{\unitlength}{\unitlength * \real{\svgscale}}%
    \fi%
  \else%
    \setlength{\unitlength}{\svgwidth}%
  \fi%
  \global\let\svgwidth\undefined%
  \global\let\svgscale\undefined%
  \makeatother%
  \begin{picture}(1,0.24763666)%
    \lineheight{1}%
    \setlength\tabcolsep{0pt}%
    \put(0,0){\includegraphics[width=\unitlength,page=1]{infty.pdf}}%
    \put(0.10049992,0.06060745){\color[rgb]{0,0,0}\makebox(0,0)[lt]{\lineheight{1.25}\smash{\begin{tabular}[t]{l}$c_\infty$\end{tabular}}}}%
    \put(0.55182515,0.16369362){\color[rgb]{0,0,0}\makebox(0,0)[lt]{\lineheight{1.25}\smash{\begin{tabular}[t]{l}$c_0$\end{tabular}}}}%
    \put(0.95610245,0.15728852){\color[rgb]{0,0,0}\makebox(0,0)[lt]{\lineheight{1.25}\smash{\begin{tabular}[t]{l}$c_0$\end{tabular}}}}%
    \put(0.59480641,0.22843207){\color[rgb]{0,0,0}\makebox(0,0)[lt]{\lineheight{1.25}\smash{\begin{tabular}[t]{l}$s_0$\end{tabular}}}}%
    \put(0.83410162,0.22853967){\color[rgb]{0,0,0}\makebox(0,0)[lt]{\lineheight{1.25}\smash{\begin{tabular}[t]{l}$s_0$\end{tabular}}}}%
    \put(0.13787168,-0.00575299){\color[rgb]{0,0,0}\makebox(0,0)[lt]{\lineheight{1.25}\smash{\begin{tabular}[t]{l}$s_\infty$\end{tabular}}}}%
    \put(0,0){\includegraphics[width=\unitlength,page=2]{infty.pdf}}%
    \put(0.71309307,0.10975763){\color[rgb]{0,0,0}\makebox(0,0)[lt]{\lineheight{1.25}\smash{\begin{tabular}[t]{l}or\end{tabular}}}}%
    \put(0,0){\includegraphics[width=\unitlength,page=3]{infty.pdf}}%
  \end{picture}%
\endgroup%

%% file: figures/crossing4.pdf_tex
%% Creator: Inkscape 1.3.2 (091e20e, 2023-11-25, custom), www.inkscape.org
%% Accompanies image file 'crossing4.pdf' (pdf, eps, ps)
\begingroup%
  \makeatletter%
  \providecommand\color[2][]{%
    \errmessage{(Inkscape) Color is used for the text in Inkscape, but the package 'color.sty' is not loaded}%
    \renewcommand\color[2][]{}%
  }%
  \providecommand\transparent[1]{%
    \errmessage{(Inkscape) Transparency is used (non-zero) for the text in Inkscape, but the package 'transparent.sty' is not loaded}%
    \renewcommand\transparent[1]{}%
  }%
  \providecommand\rotatebox[2]{#2}%
  \newcommand*\fsize{\dimexpr\f@size pt\relax}%
  \newcommand*\lineheight[1]{\fontsize{\fsize}{#1\fsize}\selectfont}%
  \ifx\svgwidth\undefined%
    \setlength{\unitlength}{310.66978635bp}%
    \ifx\svgscale\undefined%
      \relax%
    \else%
      \setlength{\unitlength}{\unitlength * \real{\svgscale}}%
    \fi%
  \else%
    \setlength{\unitlength}{\svgwidth}%
  \fi%
  \global\let\svgwidth\undefined%
  \global\let\svgscale\undefined%
  \makeatother%
  \begin{picture}(1,0.18830285)%
    \lineheight{1}%
    \setlength\tabcolsep{0pt}%
    \put(0.29887558,0.14625031){\makebox(0,0)[lt]{\lineheight{1.25}\smash{\begin{tabular}[t]{l}$a_2$\end{tabular}}}}%
    \put(0,0){\includegraphics[width=\unitlength,page=1]{crossing4.pdf}}%
    \put(0.86596461,0.08097867){\makebox(0,0)[lt]{\lineheight{1.25}\smash{\begin{tabular}[t]{l}$(0,1,-1)$\end{tabular}}}}%
    \put(0.60834127,0.08012584){\makebox(0,0)[lt]{\lineheight{1.25}\smash{\begin{tabular}[t]{l}$(0,1,-1)$\end{tabular}}}}%
    \put(0,0){\includegraphics[width=\unitlength,page=2]{crossing4.pdf}}%
    \put(0.24794538,0.08097867){\makebox(0,0)[lt]{\lineheight{1.25}\smash{\begin{tabular}[t]{l}$(0,-1,1)$\end{tabular}}}}%
    \put(-0.00194892,0.08024516){\makebox(0,0)[lt]{\lineheight{1.25}\smash{\begin{tabular}[t]{l}$(0,-1,1)$\end{tabular}}}}%
    \put(0,0){\includegraphics[width=\unitlength,page=3]{crossing4.pdf}}%
    \put(0.04080494,0.14484831){\makebox(0,0)[lt]{\lineheight{1.25}\smash{\begin{tabular}[t]{l}$a_1$\end{tabular}}}}%
    \put(0.93138126,0.14473533){\makebox(0,0)[lt]{\lineheight{1.25}\smash{\begin{tabular}[t]{l}$a_2$\end{tabular}}}}%
    \put(0.66258289,0.14484831){\makebox(0,0)[lt]{\lineheight{1.25}\smash{\begin{tabular}[t]{l}$a_1$\end{tabular}}}}%
  \end{picture}%
\endgroup%

%% file: figures/meridian.pdf_tex
%% Creator: Inkscape 1.3.2 (091e20e, 2023-11-25, custom), www.inkscape.org
%% Accompanies image file 'meridian.pdf' (pdf, eps, ps)
\begingroup%
  \makeatletter%
  \providecommand\color[2][]{%
    \errmessage{(Inkscape) Color is used for the text in Inkscape, but the package 'color.sty' is not loaded}%
    \renewcommand\color[2][]{}%
  }%
  \providecommand\transparent[1]{%
    \errmessage{(Inkscape) Transparency is used (non-zero) for the text in Inkscape, but the package 'transparent.sty' is not loaded}%
    \renewcommand\transparent[1]{}%
  }%
  \providecommand\rotatebox[2]{#2}%
  \newcommand*\fsize{\dimexpr\f@size pt\relax}%
  \newcommand*\lineheight[1]{\fontsize{\fsize}{#1\fsize}\selectfont}%
  \ifx\svgwidth\undefined%
    \setlength{\unitlength}{60.89809808bp}%
    \ifx\svgscale\undefined%
      \relax%
    \else%
      \setlength{\unitlength}{\unitlength * \real{\svgscale}}%
    \fi%
  \else%
    \setlength{\unitlength}{\svgwidth}%
  \fi%
  \global\let\svgwidth\undefined%
  \global\let\svgscale\undefined%
  \makeatother%
  \begin{picture}(1,1.7707328)%
    \lineheight{1}%
    \setlength\tabcolsep{0pt}%
    \put(0,0){\includegraphics[width=\unitlength,page=1]{meridian.pdf}}%
    \put(0.26867142,1.04137265){\color[rgb]{0,0,0}\makebox(0,0)[lt]{\lineheight{1.25}\smash{\begin{tabular}[t]{l}$\vdots$\end{tabular}}}}%
    \put(0,0){\includegraphics[width=\unitlength,page=2]{meridian.pdf}}%
    \put(0.21654074,0.36690471){\color[rgb]{0,0,0}\makebox(0,0)[lt]{\lineheight{1.25}\smash{\begin{tabular}[t]{l}$z_{i_1}$\end{tabular}}}}%
    \put(0.61280731,0.47359231){\color[rgb]{0,0,0}\makebox(0,0)[lt]{\lineheight{1.25}\smash{\begin{tabular}[t]{l}$z_{i_2}$\end{tabular}}}}%
  \end{picture}%
\endgroup%

%% file: figures/region.pdf_tex
%% Creator: Inkscape 1.0.2 (e86c8708, 2021-01-15), www.inkscape.org
%% Accompanies image file 'region.pdf' (pdf, eps, ps)
\begingroup%
  \makeatletter%
  \providecommand\color[2][]{%
    \errmessage{(Inkscape) Color is used for the text in Inkscape, but the package 'color.sty' is not loaded}%
    \renewcommand\color[2][]{}%
  }%
  \providecommand\transparent[1]{%
    \errmessage{(Inkscape) Transparency is used (non-zero) for the text in Inkscape, but the package 'transparent.sty' is not loaded}%
    \renewcommand\transparent[1]{}%
  }%
  \providecommand\rotatebox[2]{#2}%
  \newcommand*\fsize{\dimexpr\f@size pt\relax}%
  \newcommand*\lineheight[1]{\fontsize{\fsize}{#1\fsize}\selectfont}%
  \ifx\svgwidth\undefined%
    \setlength{\unitlength}{441.92755752bp}%
    \ifx\svgscale\undefined%
      \relax%
    \else%
      \setlength{\unitlength}{\unitlength * \real{\svgscale}}%
    \fi%
  \else%
    \setlength{\unitlength}{\svgwidth}%
  \fi%
  \global\let\svgwidth\undefined%
  \global\let\svgscale\undefined%
  \makeatother%
  \begin{picture}(1,0.34732742)%
    \lineheight{1}%
    \setlength\tabcolsep{0pt}%
    \put(0,0){\includegraphics[width=\unitlength,page=1]{region.pdf}}%
    \put(0.74809666,0.2894998){\makebox(0,0)[lt]{\lineheight{1.25}\smash{\begin{tabular}[t]{l}$\cdots$\end{tabular}}}}%
    \put(0,0){\includegraphics[width=\unitlength,page=2]{region.pdf}}%
    \put(0.14295125,0.25661024){\makebox(0,0)[lt]{\lineheight{1.25}\smash{\begin{tabular}[t]{l}$-$\end{tabular}}}}%
    \put(0.14528848,0.07165105){\makebox(0,0)[lt]{\lineheight{1.25}\smash{\begin{tabular}[t]{l}$-$\end{tabular}}}}%
    \put(0.33442,0.07139922){\makebox(0,0)[lt]{\lineheight{1.25}\smash{\begin{tabular}[t]{l}$+$\end{tabular}}}}%
    \put(0,0){\includegraphics[width=\unitlength,page=3]{region.pdf}}%
    \put(0.3334804,0.25922774){\makebox(0,0)[lt]{\lineheight{1.25}\smash{\begin{tabular}[t]{l}$+$\end{tabular}}}}%
    \put(0,0){\includegraphics[width=\unitlength,page=4]{region.pdf}}%
    \put(0.23264601,0.30081473){\makebox(0,0)[lt]{\lineheight{1.25}\smash{\begin{tabular}[t]{l}$\cdots$\end{tabular}}}}%
    \put(0.24025924,0.03442591){\makebox(0,0)[lt]{\lineheight{1.25}\smash{\begin{tabular}[t]{l}$\cdots$\end{tabular}}}}%
    \put(0,0){\includegraphics[width=\unitlength,page=5]{region.pdf}}%
    \put(0.74676488,0.03311552){\makebox(0,0)[lt]{\lineheight{1.25}\smash{\begin{tabular}[t]{l}$\cdots$\end{tabular}}}}%
    \put(0,0){\includegraphics[width=\unitlength,page=6]{region.pdf}}%
    \put(0.63654495,0.13868096){\makebox(0,0)[lt]{\lineheight{1.25}\smash{\begin{tabular}[t]{l}$\red{\alpha}$\end{tabular}}}}%
    \put(0.86250991,0.13262162){\makebox(0,0)[lt]{\lineheight{1.25}\smash{\begin{tabular}[t]{l}$\red{\beta}$\end{tabular}}}}%
    \put(0,0){\includegraphics[width=\unitlength,page=7]{region.pdf}}%
    \put(0.24470972,0.16734921){\makebox(0,0)[lt]{\lineheight{1.25}\smash{\begin{tabular}[t]{l}$s$\end{tabular}}}}%
  \end{picture}%
\endgroup%

%% file: figures/drawing.pdf_tex
%% Creator: Inkscape 1.0.2 (e86c8708, 2021-01-15), www.inkscape.org
%% Accompanies image file 'drawing.pdf' (pdf, eps, ps)
\begingroup%
  \makeatletter%
  \providecommand\color[2][]{%
    \errmessage{(Inkscape) Color is used for the text in Inkscape, but the package 'color.sty' is not loaded}%
    \renewcommand\color[2][]{}%
  }%
  \providecommand\transparent[1]{%
    \errmessage{(Inkscape) Transparency is used (non-zero) for the text in Inkscape, but the package 'transparent.sty' is not loaded}%
    \renewcommand\transparent[1]{}%
  }%
  \providecommand\rotatebox[2]{#2}%
  \newcommand*\fsize{\dimexpr\f@size pt\relax}%
  \newcommand*\lineheight[1]{\fontsize{\fsize}{#1\fsize}\selectfont}%
  \ifx\svgwidth\undefined%
    \setlength{\unitlength}{422.65991692bp}%
    \ifx\svgscale\undefined%
      \relax%
    \else%
      \setlength{\unitlength}{\unitlength * \real{\svgscale}}%
    \fi%
  \else%
    \setlength{\unitlength}{\svgwidth}%
  \fi%
  \global\let\svgwidth\undefined%
  \global\let\svgscale\undefined%
  \makeatother%
  \begin{picture}(1,0.34619081)%
    \lineheight{1}%
    \setlength\tabcolsep{0pt}%
    \put(0,0){\includegraphics[width=\unitlength,page=1]{drawing.pdf}}%
    \put(0.22250565,0.30194356){\makebox(0,0)[lt]{\lineheight{1.25}\smash{\begin{tabular}[t]{l}$\cdots$\end{tabular}}}}%
    \put(0.22378158,0.02329523){\makebox(0,0)[lt]{\lineheight{1.25}\smash{\begin{tabular}[t]{l}$\cdots$\end{tabular}}}}%
    \put(0,0){\includegraphics[width=\unitlength,page=2]{drawing.pdf}}%
    \put(0.74348614,0.30173029){\makebox(0,0)[lt]{\lineheight{1.25}\smash{\begin{tabular}[t]{l}$\cdots$\end{tabular}}}}%
    \put(0.74540176,0.0257169){\makebox(0,0)[lt]{\lineheight{1.25}\smash{\begin{tabular}[t]{l}$\cdots$\end{tabular}}}}%
    \put(0,0){\includegraphics[width=\unitlength,page=3]{drawing.pdf}}%
    \put(0.17896579,0.14654115){\makebox(0,0)[lt]{\lineheight{1.25}\smash{\begin{tabular}[t]{l}$p$\end{tabular}}}}%
    \put(0,0){\includegraphics[width=\unitlength,page=4]{drawing.pdf}}%
    \put(0.2948199,0.0531292){\makebox(0,0)[lt]{\lineheight{1.25}\smash{\begin{tabular}[t]{l}$q$\end{tabular}}}}%
    \put(0,0){\includegraphics[width=\unitlength,page=5]{drawing.pdf}}%
    \put(0.29700134,0.28319021){\makebox(0,0)[lt]{\lineheight{1.25}\smash{\begin{tabular}[t]{l}$q$\end{tabular}}}}%
    \put(0,0){\includegraphics[width=\unitlength,page=6]{drawing.pdf}}%
    \put(0.68788707,0.27917035){\makebox(0,0)[lt]{\lineheight{1.25}\smash{\begin{tabular}[t]{l}$p$\end{tabular}}}}%
    \put(0,0){\includegraphics[width=\unitlength,page=7]{drawing.pdf}}%
    \put(0.68882601,0.0501589){\makebox(0,0)[lt]{\lineheight{1.25}\smash{\begin{tabular}[t]{l}$p$\end{tabular}}}}%
    \put(0,0){\includegraphics[width=\unitlength,page=8]{drawing.pdf}}%
    \put(0.80395414,0.15071042){\makebox(0,0)[lt]{\lineheight{1.25}\smash{\begin{tabular}[t]{l}$q$\end{tabular}}}}%
    \put(0,0){\includegraphics[width=\unitlength,page=9]{drawing.pdf}}%
    \put(0.41306667,0.22027423){\makebox(0,0)[lt]{\lineheight{1.25}\smash{\begin{tabular}[t]{l}$\red{\beta}$\end{tabular}}}}%
    \put(0,0){\includegraphics[width=\unitlength,page=10]{drawing.pdf}}%
    \put(0.57119018,0.21459869){\makebox(0,0)[lt]{\lineheight{1.25}\smash{\begin{tabular}[t]{l}$\red{\alpha}$\end{tabular}}}}%
    \put(0,0){\includegraphics[width=\unitlength,page=11]{drawing.pdf}}%
    \put(0.14654245,0.24079994){\makebox(0,0)[lt]{\lineheight{1.25}\smash{\begin{tabular}[t]{l}$\red{\gamma}$\end{tabular}}}}%
    \put(0,0){\includegraphics[width=\unitlength,page=12]{drawing.pdf}}%
    \put(0.84051833,0.09332334){\makebox(0,0)[lt]{\lineheight{1.25}\smash{\begin{tabular}[t]{l}$\red{\gamma}$\end{tabular}}}}%
  \end{picture}%
\endgroup%

%% file: figures/example2.pdf_tex
%% Creator: Inkscape 1.3.2 (091e20e, 2023-11-25, custom), www.inkscape.org
%% Accompanies image file 'example2.pdf' (pdf, eps, ps)
\begingroup%
  \makeatletter%
  \providecommand\color[2][]{%
    \errmessage{(Inkscape) Color is used for the text in Inkscape, but the package 'color.sty' is not loaded}%
    \renewcommand\color[2][]{}%
  }%
  \providecommand\transparent[1]{%
    \errmessage{(Inkscape) Transparency is used (non-zero) for the text in Inkscape, but the package 'transparent.sty' is not loaded}%
    \renewcommand\transparent[1]{}%
  }%
  \providecommand\rotatebox[2]{#2}%
  \newcommand*\fsize{\dimexpr\f@size pt\relax}%
  \newcommand*\lineheight[1]{\fontsize{\fsize}{#1\fsize}\selectfont}%
  \ifx\svgwidth\undefined%
    \setlength{\unitlength}{112.69643468bp}%
    \ifx\svgscale\undefined%
      \relax%
    \else%
      \setlength{\unitlength}{\unitlength * \real{\svgscale}}%
    \fi%
  \else%
    \setlength{\unitlength}{\svgwidth}%
  \fi%
  \global\let\svgwidth\undefined%
  \global\let\svgscale\undefined%
  \makeatother%
  \begin{picture}(1,1.73933022)%
    \lineheight{1}%
    \setlength\tabcolsep{0pt}%
    \put(0,0){\includegraphics[width=\unitlength,page=1]{example2.pdf}}%
    \put(0.55542858,1.12494998){\makebox(0,0)[lt]{\lineheight{1.25}\smash{\begin{tabular}[t]{l}$x_1$\end{tabular}}}}%
    \put(0.31742239,1.35764108){\makebox(0,0)[lt]{\lineheight{1.25}\smash{\begin{tabular}[t]{l}$\infty$\end{tabular}}}}%
    \put(0.21771932,1.13936218){\makebox(0,0)[lt]{\lineheight{1.25}\smash{\begin{tabular}[t]{l}$x_5$\end{tabular}}}}%
    \put(0.21433772,0.65365184){\makebox(0,0)[lt]{\lineheight{1.25}\smash{\begin{tabular}[t]{l}$x_2$\end{tabular}}}}%
    \put(0.32233133,0.88843544){\makebox(0,0)[lt]{\lineheight{1.25}\smash{\begin{tabular}[t]{l}$x_3$\end{tabular}}}}%
    \put(0.54525645,0.6704513){\makebox(0,0)[lt]{\lineheight{1.25}\smash{\begin{tabular}[t]{l}$x_4$\end{tabular}}}}%
    \put(0,0){\includegraphics[width=\unitlength,page=2]{example2.pdf}}%
    \put(0.34384024,0.36448467){\makebox(0,0)[lt]{\lineheight{1.25}\smash{\begin{tabular}[t]{l}$0$\end{tabular}}}}%
    \put(0,0){\includegraphics[width=\unitlength,page=3]{example2.pdf}}%
    \put(-0.00231299,0.37609242){\makebox(0,0)[lt]{\lineheight{1.25}\smash{\begin{tabular}[t]{l}$1$\end{tabular}}}}%
    \put(-0.00537257,0.72994119){\makebox(0,0)[lt]{\lineheight{1.25}\smash{\begin{tabular}[t]{l}$1$\end{tabular}}}}%
    \put(-0.00338438,1.20541338){\makebox(0,0)[lt]{\lineheight{1.25}\smash{\begin{tabular}[t]{l}$1$\end{tabular}}}}%
    \put(0.86953299,0.85737481){\makebox(0,0)[lt]{\lineheight{1.25}\smash{\begin{tabular}[t]{l}$1$\end{tabular}}}}%
  \end{picture}%
\endgroup%

%% file: figures/crossing3.pdf_tex
%% Creator: Inkscape 1.0.2 (e86c8708, 2021-01-15), www.inkscape.org
%% Accompanies image file 'crossing3.pdf' (pdf, eps, ps)
\begingroup%
  \makeatletter%
  \providecommand\color[2][]{%
    \errmessage{(Inkscape) Color is used for the text in Inkscape, but the package 'color.sty' is not loaded}%
    \renewcommand\color[2][]{}%
  }%
  \providecommand\transparent[1]{%
    \errmessage{(Inkscape) Transparency is used (non-zero) for the text in Inkscape, but the package 'transparent.sty' is not loaded}%
    \renewcommand\transparent[1]{}%
  }%
  \providecommand\rotatebox[2]{#2}%
  \newcommand*\fsize{\dimexpr\f@size pt\relax}%
  \newcommand*\lineheight[1]{\fontsize{\fsize}{#1\fsize}\selectfont}%
  \ifx\svgwidth\undefined%
    \setlength{\unitlength}{310.66978635bp}%
    \ifx\svgscale\undefined%
      \relax%
    \else%
      \setlength{\unitlength}{\unitlength * \real{\svgscale}}%
    \fi%
  \else%
    \setlength{\unitlength}{\svgwidth}%
  \fi%
  \global\let\svgwidth\undefined%
  \global\let\svgscale\undefined%
  \makeatother%
  \begin{picture}(1,0.18830285)%
    \lineheight{1}%
    \setlength\tabcolsep{0pt}%
    \put(0.14255608,0.01524228){\makebox(0,0)[lt]{\lineheight{1.25}\smash{\begin{tabular}[t]{l}$r_i$\end{tabular}}}}%
    \put(0.21492038,0.01549366){\makebox(0,0)[lt]{\lineheight{1.25}\smash{\begin{tabular}[t]{l}$r_j$\end{tabular}}}}%
    \put(0,0){\includegraphics[width=\unitlength,page=1]{crossing3.pdf}}%
    \put(0.76298942,0.01524228){\makebox(0,0)[lt]{\lineheight{1.25}\smash{\begin{tabular}[t]{l}$r_i$\end{tabular}}}}%
    \put(0.83535372,0.01549366){\makebox(0,0)[lt]{\lineheight{1.25}\smash{\begin{tabular}[t]{l}$r_j$\end{tabular}}}}%
    \put(0.86596461,0.08097867){\makebox(0,0)[lt]{\lineheight{1.25}\smash{\begin{tabular}[t]{l}$(1,0,-1)$\end{tabular}}}}%
    \put(0.60834127,0.08012584){\makebox(0,0)[lt]{\lineheight{1.25}\smash{\begin{tabular}[t]{l}$(-1,1,0)$\end{tabular}}}}%
    \put(0,0){\includegraphics[width=\unitlength,page=2]{crossing3.pdf}}%
    \put(0.24794538,0.08097867){\makebox(0,0)[lt]{\lineheight{1.25}\smash{\begin{tabular}[t]{l}$(1,-1,0)$\end{tabular}}}}%
    \put(-0.00194892,0.08024516){\makebox(0,0)[lt]{\lineheight{1.25}\smash{\begin{tabular}[t]{l}$(-1,0,1)$\end{tabular}}}}%
    \put(0,0){\includegraphics[width=\unitlength,page=3]{crossing3.pdf}}%
  \end{picture}%
\endgroup%

%% file: figures/cup.pdf_tex
%% Creator: Inkscape 1.0.2 (e86c8708, 2021-01-15), www.inkscape.org
%% Accompanies image file 'cup.pdf' (pdf, eps, ps)
\begingroup%
  \makeatletter%
  \providecommand\color[2][]{%
    \errmessage{(Inkscape) Color is used for the text in Inkscape, but the package 'color.sty' is not loaded}%
    \renewcommand\color[2][]{}%
  }%
  \providecommand\transparent[1]{%
    \errmessage{(Inkscape) Transparency is used (non-zero) for the text in Inkscape, but the package 'transparent.sty' is not loaded}%
    \renewcommand\transparent[1]{}%
  }%
  \providecommand\rotatebox[2]{#2}%
  \newcommand*\fsize{\dimexpr\f@size pt\relax}%
  \newcommand*\lineheight[1]{\fontsize{\fsize}{#1\fsize}\selectfont}%
  \ifx\svgwidth\undefined%
    \setlength{\unitlength}{154.4999863bp}%
    \ifx\svgscale\undefined%
      \relax%
    \else%
      \setlength{\unitlength}{\unitlength * \real{\svgscale}}%
    \fi%
  \else%
    \setlength{\unitlength}{\svgwidth}%
  \fi%
  \global\let\svgwidth\undefined%
  \global\let\svgscale\undefined%
  \makeatother%
  \begin{picture}(1,0.56230423)%
    \lineheight{1}%
    \setlength\tabcolsep{0pt}%
    \put(0,0){\includegraphics[width=\unitlength,page=1]{cup.pdf}}%
    \put(0.13530472,-0.02627781){\makebox(0,0)[lt]{\lineheight{1.25}\smash{\begin{tabular}[t]{l}$\infty$\end{tabular}}}}%
    \put(0,0){\includegraphics[width=\unitlength,page=2]{cup.pdf}}%
    \put(0.17233007,0.05987681){\makebox(0,0)[lt]{\lineheight{1.25}\smash{\begin{tabular}[t]{l}(a)\end{tabular}}}}%
    \put(0,0){\includegraphics[width=\unitlength,page=3]{cup.pdf}}%
    \put(0.00483383,0.42880602){\makebox(0,0)[lt]{\lineheight{1.25}\smash{\begin{tabular}[t]{l}(b)\end{tabular}}}}%
    \put(0.39563113,0.37055661){\makebox(0,0)[lt]{\lineheight{1.25}\smash{\begin{tabular}[t]{l}$\cdots$\end{tabular}}}}%
    \put(0,0){\includegraphics[width=\unitlength,page=4]{cup.pdf}}%
    \put(0.10986914,0.24837915){\makebox(0,0)[lt]{\lineheight{1.25}\smash{\begin{tabular}[t]{l}$\vdots$\end{tabular}}}}%
  \end{picture}%
\endgroup%